\documentclass[12pt,a4paper]{amsart}
\usepackage {amsmath, amssymb, a4wide, epsfig, enumerate, psfrag, color}
\usepackage[latin1]{inputenc}
\usepackage[english]{babel}
\usepackage[active]{srcltx}

\date{\today}

\author{Bertrand Deroin \and Romain Dujardin}
\thanks{B.D.'s research was partially supported by ANR-08-JCJC-0130-01,  ANR-09-BLAN-0116. 
}
\address{CNRS \\ D\'epartement de Math\'ematique d'Orsay \\ B\^atiment 425,
Universit\'e de Paris Sud,
91405 Orsay cedex, France.}
\email{bertrand.deroin@math.u-psud.fr}
\address{LAMA  \\
Universit\'e Paris-Est Marne-la-Vall\'ee  \\
5 boulevard Descartes \\
77454 Champs sur Marne  \\
France}
\email{romain.dujardin@univ-mlv.fr}

\title[Complex projective structures]{Complex projective structures: 
Lyapunov exponent, degree and harmonic measure}

\newcommand{\cc}{\mathbb{C}}
\newcommand{\bb}{\mathbb{B}}
\newcommand{\rr}{\mathbb{R}}
\newcommand{\dd}{\mathbb{D}}
\newcommand{\zz}{\mathbb{Z}}
\newcommand{\nn}{\mathbb{N}}
\newcommand{\pp}{\mathbb{P}}
\newcommand{\ee}{\mathbb{E}}
\newcommand{\hh}{\mathbb{H}}
 
\newcommand{\e}{\varepsilon}
\newcommand{\cv}{\rightarrow}

\newcommand{\fr}{\partial}
\newcommand{\om}{\Omega}
\newcommand{\set}[1]{\left\{#1\right\}}
\newcommand{\norm}[1]{\left\Vert#1\right\Vert}
\newcommand{\abs}[1]{\left\vert#1\right\vert}
\newcommand{\cd}{{\cc^2}}

\newcommand{\pu}{{\mathbb{P}^1}}
\newcommand{\rest}[1]{ \arrowvert_{#1}}

\newcommand{\unsur}[1]{\frac{1}{#1}}

\renewcommand{\o}{\omega}
\newcommand{\lrpar}[1]{\left(#1\right)}
\newcommand{\bra}[1]{\left\langle #1\right\rangle}

\newcommand{\la}{\lambda}
\newcommand{\lo}{{\lambda_0}}
 
\newcommand{\geom}{\dot \wedge} 
 \newcommand{\dev}{\mathsf{dev}}
 \newcommand{\hol}{\mathsf{hol}}
\newcommand{\La}{\Lambda}

\newcommand{\PSL}{\mathrm{PSL}(2,\mathbb C)}
\newcommand{\PSLR}{\mathrm{PSL}(2,\mathbb R)}

\newcommand{\tbif}{T_{\mathrm{bif}}}

\newcommand{\note}[1]{\marginpar{\tiny #1}}
\newcommand{\itm}{\item[-]}

\DeclareMathOperator{\supp}{Supp}
\DeclareMathOperator{\vol}{vol}
\DeclareMathOperator{\Int}{Int}

\DeclareMathOperator{\tr}{tr}

\newtheorem{prop} {Proposition} [section]

\newtheorem{lem}[prop] {Lemma}
\newtheorem{cor}[prop]{Corollary}
\newtheorem{theo}{Theorem}
\newtheorem{coro}[theo]{Corollary}

\newtheorem{defprop}[prop]{Definition-Proposition}

\theoremstyle{remark}

\newtheorem{rmk}[prop]{Remark}

\begin{document}

\begin{abstract}
We study several new invariants associated to a holomorphic projective structure on a Riemann surface of finite analytic type:
the Lyapunov exponent of its holonomy which is of probabilistic/dynamical nature and was introduced in our previous work; 
the degree which measures the asymptotic covering rate of the developing map; and a family of harmonic measures on the Riemann sphere, 
previously introduced by Hussenot. We show that the degree and the Lyapunov exponent are related by a simple formula and give estimates for the Hausdorff dimension of the harmonic measures in terms of the Lyapunov exponent. 
In accordance with the famous ``Sullivan dictionary", this leads to a description of the space of such projective structures that is reminiscent of that of the space of polynomials in holomorphic dynamics.
\end{abstract}

\maketitle

 \section*{Introduction}
Our purpose in this paper is to introduce several new objects associated with a $\mathbb{CP}^1$-structure  
on a Riemann surface  $X$ of finite type, and study their relationships. In the non-compact case, we assume that the projective 
structure is ``parabolic at the cusps", in a sense that will be made precise below. 

To such a projective structure $\sigma$, we associate: 
\begin{itemize}
\itm a {\em Lyapunov exponent} $\chi(\sigma)$, which was constructed in \cite{Bers1};
\itm a {\em degree} $\deg(\sigma)$ which is simply the (normalized) asymptotic covering degree of the  developing map 
$\widetilde X\cv \pu$ ($\widetilde X$ the universal cover of $X$);
\itm a family of {\em harmonic measures} $(\nu_x)_{x\in \widetilde X}$ on $\pu$, which generalize the traditional harmonic measures on the limit sets of Kleinian groups (throughout the paper, $\pu$ stands for $\cc\pu$).
\end{itemize}
 We will show that the Lyapunov exponent and the degree   are related by a simple formula, and give estimates for the Hausdorff dimension of the harmonic measures in terms of $\chi$.  We give several applications of these ideas to the study of the space $P(X)$ of (parabolic) projective structures on $X$, in particular revealing new aspects of the famous Sullivan dictionary between rational and 
 M\"obius dynamics on $\pu$. 

Before proceeding to a detailed presentation of these results, let us present the main characters  of the story.  
 
\subsection{Parabolic $\mathbb P^1$-structures}\label{ss:parabolic type PS}
Let us fix   a Riemann surface $X$ of finite type (genus $g$ with $n$ punctures)
with negative Euler characteristic $\mathrm{eu(X)}=2-2g-n$, 
as well as  a universal cover  $c : \widetilde{X}\rightarrow X$. 
  Throughout this paper,  by definition the (unmarked) fundamental group $\pi_1(X)$ is  the deck group of this covering.
Recall that the automorphism group of the Riemann sphere is the group 
$\text{PSL} (2,\mathbb C)$ acting on $\mathbb  P^1 = \mathbb C \cup \set{\infty} $ by the formula 
$ \lrpar{\begin{smallmatrix}
   a & b \\ c& d                                                                                                                                                                                                                                                                                                                                                                                                                                                                                                                                                                                                                                                      \end{smallmatrix}} \cdot z = \frac{az + b }{cz + d}$. By definition a {\em Kleinian group} is a discrete subgroup of $\PSL$. 
   
It is convenient to define a  $\pu$-structure on $X$ 
in terms of the so-called development-holonomy pair $(\mathsf{dev},\mathsf{hol})$. 
Consider a non constant locally injective meromorphic map  $\mathsf{dev} : \widetilde{X} \rightarrow \mathbb P ^1 $, satisfying the equivariance property $\mathsf{dev}\circ \gamma  = \mathsf{hol}(\gamma)\circ \mathsf{dev} $, where  
 $\mathsf{hol}$ is a representation $\pi_1(X)\cv\text{PSL} (2,\mathbb C)$. 
If $A\in \PSL$, the pairs $(\mathsf{dev},\mathsf{hol})$ and $(A\circ \mathsf{dev},A \circ \mathsf{hol}\circ A^{-1})$ will be declared as equivalent. By definition, a
\textit{$\mathbb P^1 $-structure} is an equivalence class of such pairs.
 We refer here to the survey paper by Dumas, see~\cite{dumas} for a 
 comprehensive   treatment of 
this notion. 

When the surface $X$ is not compact (hence by assumption it is biholomorphic to a compact Riemann surface punctured at a finite set),  we restrict ourselves to the subclass of \textit{parabolic} $\mathbb P^1$-structures. Such a structure has the following well-defined local model around the punctures: each 
 puncture has a neighborhood which is {\em projectively}
  equivalent to the quotient of  the upper half plane
   by the translation $z\mapsto z+1$.  

For instance, the canonical projective structure $\sigma_{\rm Fuchs}$ induced by uniformization (i.e. viewing $X$ as a quotient of $\hh$ under a Fuchsian group) is of this type. 
More generally, the proof of the Ahlfors finiteness theorem (see  \cite[Lemma 1]{ahlfors}) shows that if $\Gamma$ 
is any torsion free Kleinian group, and $\om$ is the orbit of a discontinuity component, then the induced 
 projective structure on $\om/\Gamma$ is parabolic. These have been known as {\em covering projective structures}, because  
 the developing map is a covering onto its image in this case \cite{kra1, kra2}. An important example is given  
by quasi-Fuchsian deformations of the canonical structure $\sigma_{\rm Fuchs}$.

There are many other examples of parabolic $\mathbb P^1$-structures. For instance surgery operations such as grafting (see Hejhal's 
original construction in \cite{hejhal grafting}) may produce a parabolic $\mathbb P^1$-structure with holonomy a Kleinian group  that is not of covering type. Such projective  structures are usually called 
\textit{exotic}. 
There are yet other examples of $\mathbb P^1$-structures: a remarkable theorem of Gallo, Kapovich and Marden \cite{gkm}
asserts that (when $X$ is compact) a  representation $\pi_1(X) \rightarrow \text{PSL}(2,\mathbb C)$ is the holonomy of a $\mathbb P^1$-structure on $X$ (for some Riemann surface structure on $X$) if and only if it is non elementary and it lifts to
 a representation with values in $\text{SL}(2,\mathbb C)$. In particular, there exist  $\mathbb P^1$-structures with holonomy a non discrete (or even dense) subgroup of $\text{PSL} (2,\mathbb C)$.

It is  also  of interest  to deal with  the case of {\em branched $\pu$-structures}, 
  where the local injectivity assumption on the developing map is dropped.  Examples include   conical metrics
   of constant curvature equal to $1,0$ or $-1$, with conical points of angle  multiple of $2\pi$. It 
turns out that some of the results in the paper carry over to this setting. 
The necessary adaptations will be explained in Appendix \ref{sec:appendix}. 

\subsection{The degree} 
Let $\hh$ denote the upper half-plane $\set{\tau, \Im(\tau)>0}$, and 
  $g_P = \frac{2|d\tau|}{\Im \tau}$ be the \textit{Poincar\'e} (or {\em hyperbolic}) {\em  metric} on $\hh$, that is
the  unique complete conformal metric of curvature $-1$,  
which is invariant under $\mathrm{Aut}(\mathbb H)\simeq \mathrm{PSL}(2, \rr)$.
Taking the pull-back of this metric under any biholomorphism between $\widetilde{X}$ and $\mathbb H$ 
endows $\widetilde{X}$ with a complete conformal metric of curvature $-1$ invariant under $\pi_1(X)$, 
therefore this metric descends to $X$.  It is well known that when $X$ is of finite type, the hyperbolic metric has finite volume. 

Recall that a representation $\pi_1(X) \rightarrow \text{PSL} (2, \mathbb C)$ is \text{non elementary} if it does not preserve any 
probability measure on the Riemann sphere. The holonomy of a parabolic projective structure 
is always non elementary: see \cite[Theorem 11.6.1, p. 695]{gkm} for the compact case, and \cite[Lemma 10]{cdfg} for a proof in the case of the fourth punctured sphere, which readily extends to all punctured surfaces. 

\medskip

If $\sigma$ is a parabolic projective structure, we want to define $\delta(\sigma)$ as a nonnegative
 number counting the average asymptotic covering degree of 
$\mathsf{dev}_\sigma : \widetilde X \cv \pu$. 
For any $x\in \widetilde{X}$ we denote by $B(x,R)$ the ball centered at $x$ of radius $R$ in the Poincar\'e metric, 
and by $\vol$ the hyperbolic volume.

\begin{defprop} \label{defprop:degree}
Let $X$ be a Riemann surface of finite type and $\sigma$ be a parabolic projective structure on $X$. Choose a 
 developing map  $\mathsf{dev}: \widetilde{X} \rightarrow \mathbb P^1$. 
 Let $(x_n)$ be a sequence of points in $\widetilde{X}$ whose projections $c (x_n)$ stay in  a compact subset of $X$,   $(R_n)$ be a sequence of radii tending to infinity, and $(z_n)$ be an arbitrary sequence in $\pu$. Then the   limit  
\begin{equation}\label{eq:limit}  \delta = \lim_{n\rightarrow \infty} \frac{ \#B(x_n,R_n)\cap \dev^{-1}(z_n)}{\mathrm{vol}(B(x_n,R_n))} \end{equation}
exists, and does not depend on the chosen sequences $(x_n)$, $(R_n)$ nor on the developing map $\mathsf{dev}$. 
The number $\mathrm{deg}(\sigma)=  \mathrm{vol} (X) \delta$ is by definition the
{\em degree} of the projective structure.  
\end{defprop}

The existence of the limit in~\eqref{eq:limit}   is not obvious, in particular due to the possibility of boundary effects. The proof ultimately relies on a result of Bonatti and G\'omez-Mont \cite{bgm} and will be carried out  in \S\ref{ss:existence of degree}. 
Observe that this result is reminiscent from Nevanlinna theory, though the information we obtain is much more precise.
We can actually derive the asymptotics of the Nevanlinna theoretic 
 counting function $N(r,\dev,z)$ and characteristic $T(r, \dev)$ of the developing map
(see \cite{nevanlinna}) and show that these quantities are 
governed by the degree. Namely, for every $z\in \mathbb P^1$  an easy computation shows that
\begin{equation} \label{eq:nevanlinna characteristic} N(r,\dev,z) \underset{r\rightarrow \infty}\sim  T(r,  \dev ) \underset{r\rightarrow 1}\sim 2\pi \delta \log\big(\frac{1}{1-r}\big).\end{equation}
Besides, Nevanlinna theory is known to have connections with Brownian motion, see \cite{carne}. In this paper we will explore this relationship from a different point of view.

The     reason for introducing the normalized invariant $\mathrm{deg}(\sigma)$ is that $\delta(\sigma)$ is invariant under finite coverings, hence does not behave like a degree. 

 We also show that projective structures with vanishing degree 
are exactly the covering projective structures (Proposition \ref{p:vanishing}).

\subsection{The Lyapunov exponent} \label{ss:lyapunov}
The second invariant,   the \textit{Lyapunov exponent} of a parabolic projective structure was defined in our previous work~\cite{Bers1}. It depends only on the holonomy $\mathsf{hol}_\sigma$ of the structure, and also on the induced Riemann surface structure on $X$. Fix a basepoint $\star\in X$, in particular an identification  between
  the covering group $\pi_1(X)$ and the usual fundamental group $\pi_1(X,\star)$. 
  As $X$ is endowed with its Poincar\'e metric,  Brownian motion on $X$ is well-defined. Throughout
   the paper, {\em Brownian motion} will refer to the stochastic process with continuous time whose infinitesimal generator is the hyperbolic laplacian (instead of $\unsur{2}\Delta$, which is another usual convention). 
  Let $W_\star$ be the Wiener measure on the set of continuous paths $\omega: [0,\infty) \rightarrow X$ 
  starting at $\omega(0)= \star$.

 \begin{defprop}\label{def:lyap}
 Let $X$ and $\sigma$ be as above. 
Define a family of loops as follows: for $t>0$, consider a Brownian path $\o$ issued from $\star$, and concatenate 
$\o\rest{[0, t]}$ with a shortest geodesic joining $\o(t)$ and $\star$, 
thus obtaining a closed loop $\widetilde \o_t$. Then for $W_\star$ a.e. Brownian path $\o$ the limit 
\begin{equation}\label{eq:deflyap}
\chi (\sigma) = \lim_{t\cv\infty} \unsur{t}\log \norm{\mathsf{hol}\lrpar{\widetilde \o_t}}
\end{equation}
 exists and does not depend on $\o$.  
This number is by definition the Lyapunov exponent of $\sigma$.
 \end{defprop}

Here $\norm{\cdot}$ is any matrix norm on $\PSL$.
The existence of  the limit in \eqref{eq:deflyap} was established in \cite[Def-Prop. 2.1]{Bers1}. As expected it is 
 a consequence of the subadditive ergodic theorem.  With  notation as in  \cite{Bers1}, 
 $\chi(\sigma)  = \chi_{\rm Brown} (\mathsf{hol})$. 
 Another way to define $\chi (\sigma)$ goes as follows (see \cite[Rmk 3.7]{Bers1}:   
  identify  $\pi_1(X)$ with a Fuchsian group $\Gamma$ and 
   independently   random  elements  $\gamma_n  \in \Gamma\cap B_{\hh}(0, R_n)$,  relative to the counting measure.   
   Here $(R_n)$ is a sequence tending to infinity  as fast as, say $n^\alpha$ for $\alpha>0$. 
 Then almost surely  
 $$\unsur{d_\hh(0, \gamma_n(0))  }\log \norm{\mathsf{hol}(\gamma_n)} \underset{n\cv\infty}\longrightarrow \chi(\sigma).$$

\subsection{The harmonic measures}\label{ss:harmonic measure} 
The third object that we associate to a $\mathbb P^1$-structure on $X$ is 
a family of \textit{harmonic measures}  $\{ \nu_x \} _{x\in \widetilde{X}}$ on the Riemann sphere, indexed by $\widetilde X$.
It can be defined in    several ways. The following appealing presentation was introduced by 
  Hussenot in his PhD thesis \cite{hussenot}: 

\begin{defprop}[Hussenot] \label{defprop:harmonic measure}
Let $X$ be a Riemann surface of finite type and $\sigma$ be a parabolic projective structure on $X$. Choose a 
representing pair $(\mathsf{dev}, \mathsf{hol})$. 
 Then for every $x\in \widetilde X$, and $W_x$ a.e. Brownian path starting at $\omega(0)=x$, 
 there exists a point $\mathrm e(\omega)$ on   $\pu$ defined by the property  that 
\[   \frac{1}{t} \int _0 ^t \mathsf{dev}_*\lrpar{\delta _{\omega (s)}} ds  \underset{t\rightarrow +\infty}\longrightarrow  \delta_{\mathrm e(\omega)} . \]
The distribution of the point $\mathrm e(\omega )$ subject to the condition that $\omega(0)= x$ is the measure $\nu_x$.
\end{defprop}

For covering $\pu$-structures, we recognize the classical harmonic measures on the  limit set.  

\medskip

Another definition of the harmonic measures is based on the so-called Furstenberg boundary map, 
which was designed in \cite{furstenberg},
  based on the discretization of Brownian motion in $\hh$
   (see also \cite[Theorem 3]{margulis} for a different approach). 
Furstenberg shows that if $\Gamma\subset \text{PSL}(2,\mathbb R)$ is a cofinite Fuchsian group and 
$\rho:\Gamma\cv\PSL$ is a non-elementary representation, there exists  a unique measurable
 equivariant mapping $\theta: \mathbb P^1(\mathbb R)\cv\pu$ 
 defined  a.e. with respect to Lebesgue measure
  (here $\mathbb P^1(\mathbb R)$ is viewed as $\fr\hh$)
  Choose a biholomorphism
   $\widetilde X\simeq \hh$, thereby identifying  $\pi_1(X)$ with a cofinite Fuchsian group. 
   For   $\tau\in \hh$,  recall that  the classical harmonic measure $m_\tau$ is 
   a probability measure with smooth density on $\mathbb P^1(\mathbb R)$, defined as
the exit distribution of Brownian paths issued from $\tau$. 
The harmonic measure $\nu_x$ is then   defined by   $\nu_x = \theta_*m_\tau$, where
 $x\in \widetilde{X}$ corresponds to $\tau \in \hh$, and $\theta$ is associated to $\dev$. 
From this perspective it is clear that, the measures $\nu_x$ are mutually absolutely continuous and depend harmonically on $x$. 

\subsection{The main results}
The  main result in this paper  is the following formula, relating the Lyapunov exponent and the degree of a $\pu$-structure. 

\begin{theo}\label{theo:formula}
Let $\sigma$ be a parabolic holomorphic $\mathbb{P}^1$ structure on a hyperbolic 
 Riemann surface $X$ of finite type. Let as above  
$\chi(\sigma)$,   $\delta(\sigma)$, and $\deg(\sigma)$   respectively denote the Lyapunov exponent, the unnormalized degree and the 
degree of $\sigma$. Then  the following formula  holds:
 \begin{equation}\label{eq:formula}
 \displaystyle \chi (\sigma) = \frac{1}{2} + 2\pi\delta(\sigma)    = \frac12+ \frac{\deg(\sigma)}{\abs{\mathrm{eu(X)}}}.
 \end{equation}
 \end{theo}

Surprisingly enough, the proof is based on the ergodic theory of holomorphic foliations.
 Indeed,
 to any representation $\rho: \pi_1(X)\cv \PSL$ one  classically associates its {\em suspension}, a
  flat   $\mathbb P^1$-bundle $M_\sigma \rightarrow X$ whose monodromy is $\rho$. 
In more concrete terms, it is the quotient of $\widetilde X\times \pu$ under the diagonal action of $\pi_1(X)$. The horizontal foliation of $
\widetilde X\times \pu$ descends to a holomorphic foliation on $M_\sigma$  transverse to the $\pu$ fibers, with monodromy $\rho$. This  ``dictionary" between $\pu$-structures and transverse sections of flat $\pu$-bundles, 
  was investigated e.g. in \cite{loray marin}.

For $\rho = \mathsf{hol}_\sigma$, we 
analyze this foliation from the point of view of Garnett's theory of foliated harmonic measures and currents.  
This interplay   was already explored by Bonatti and Gomez-M\'ont \cite{bgm} and Alvarez \cite{alvarez}. 
The key of the proof of the theorem  is to interpret $\chi$ and $\delta$ as cohomological quantities on 
$M_{\sigma}$. The idea that foliated Lyapunov exponents can be computed in cohomology  
stems from  the first author's thesis   (see \cite[Appendice]{deroin levi plate}). 

When $X$  is not compact, to prove the result we  compactify both $M$ and the 
foliation.  The computations then become much more delicate because the compactified foliation is singular. The details are carried out in 
Sections \ref{s:degree}, \ref{sec:lyapunov} and \ref{sec:proof}. 


\medskip

Theorem \ref{theo:formula} is mostly interesting for the purpose  of studying the space of projective structures on $X$.  Let $P(X)$ be the 
space of parabolic projective structures on $X$ which are compatible with the complex structure. It is well known that $P(X)$ is naturally 
isomorphic to an affine space of quadratic differentials on $X$ of dimension $3g-3+n$ (see \S \ref{subs:Teichmuller} 
below for more details). The Bers simultaneous uniformization theorem implies that the Teichm\"uller space of marked conformal structures 
on $X$  embeds as a bounded open subset $B(X)\subset P(X)$ (the {\em Bers slice}), whose geometry has been extensively studied. 

The Sullivan dictionary is a very fruitful set of analogies between the dynamics of rational transformations on $\pu$ and the theory of 
Kleinian (and more generally M\"obius) groups. 
In \cite{mcm renormalization}, McMullen draws a fundamental parallel between the Bers slice  $B(X)$ and the Mandelbrot set. 
We take one step further here by relating  $P(X)$ and  the space of polynomials. 
In this respect,  Theorem 
 \ref{theo:formula}  should be understood as the analogue of the familiar Manning-Przytycki 
 formula \cite{manning, prz} for the Lyapunov
  exponent of the maximal entropy measure of a polynomial. This analogy should be used  as a guide for the forthcoming results. 

It was shown in \cite{Bers1} that $\sigma\mapsto \chi(\sigma)$ is a continuous (H\"older) plurisubharmonic (psh for short) 
function on $P(X)$, hence it follows from Theorem \ref{theo:formula} that $\deg$ is continuous and  psh, too. 
In addition  we see that $\chi(\sigma)$ reaches its minimal value $\frac12$ exactly when $\deg(\sigma)=0$. 
As already observed, $\deg =0$ on $\overline{B(X)}$, so in particular $\chi = \frac12$ there.

\medskip

%

A first result which parallels exactly the dynamics of polynomials  concerns the Hausdorff dimension of the harmonic measures. 

\begin{theo}\label{theo:dimension}
Let $X$ be a hyperbolic 
Riemann surface of finite type and $\sigma$ be a parabolic projective structure on $X$. Let as above $\chi$ be its Lyapunov 
exponent and  $(\nu_x)_{x\in \widetilde X}$ be the associated family of harmonic measures. 
Then for every $x$, $$\dim_H(\nu_x) \leq \frac{1}{2\chi}\leq 1.$$ Furthermore $\dim_H(\nu_x) = 1$ if and only if 
$\sigma$ belongs to the closure of the Bers slice ${B(X)}$. 
\end{theo}

Notice that since the measures $\nu_x$ are mutually absolutely continuous,  $\dim_H(\nu_x)$ is independent of $x$, so abusing notation, we often simply denote it as $\dim_H(\nu)$. 
The proof is an adaptation of    Ledrappier  \cite[Thm 1]{ledrappier}.

So, as in the polynomial case,   Theorem \ref{theo:dimension} provides  an alternate approach 
to the classical bound $\dim_H(\nu)\leq 1$ for  the harmonic measure on boundary of discontinuity components of finitely generated Kleinian groups, 
which follows from the  famous results of Makarov \cite{makarov} and  Jones-Wolff \cite{jones wolff}. 
In addition, with this method we are also able to show  that $\dim_H(\nu)< 1$ when the component is not simply connected.
Indeed we have the more precise bound $\dim_H(\nu)\leq \frac{A}{2\chi}$, where $0\leq A\leq 1$ 
is an invariant of the flat foliation, and $A< 1$ when $\hol$ is not injective. 
 
 We also see that the value of the  dimension  of the harmonic measures detects exotic quasifuchsian structures, that is, 
projective structures with quasifuchsian holonomy which do not belong to the Bers slice. 

\medskip

As a third application of Theorem \ref{theo:formula}, we recover a result due to Shiga \cite{shiga}.

 \begin{theo} \label{theo:convex}
 Let $X$ be a hyperbolic Riemann surface of finite type (of genus $g$ with $n$ punctures). 
 The closure of the Bers embedding  $B(X)$  
 is a polynomially convex compact subset of the space $P(X)\simeq \cc^{3g-3+n}$ of holomorphic projective structures on $X$. 
 As a consequence, $B(X)$ is a polynomially convex (or Runge) domain.
 \end{theo}
 
Recall that   a  compact set $K$ in $\cc^N$ is polynomially convex if $\widehat K = K$, where 
$$\widehat{K} = \set{z\in \cc^N, \ \abs{P(z)}\leq \sup_K \abs{P} \text{ for every  polynomial }P}.$$
An open set $U\subset \cc^N$ is said to be  polynomially convex (or Runge) if for every $K\Subset U$, 
$\widehat K\subset U$. 
The theorem  may be reformulated 
 by saying that $\overline{B(X)}$ is defined by countably many polynomial inequalities of the form 
$\abs{P}\leq 1$. This is not an intrinsic property of Teichm\"uller space, but rather a property of its embedding 
into the space $P(X)$ of holomorphic projective structures on $X$
 (as opposed to the Bers-Ehrenpreis  theorem that Teichm\"uller  spaces are holomorphically convex). 
 
 Shiga's proof is  based on the Grunsky inequality on univalent functions. 
 Only the polynomial convexity of $B(X)$ was asserted in \cite{shiga}, but the  proof 
covers the case of $\overline{B(X)}$ as well.  Our approach is based on the 
elementary fact that the minimum locus   of a global psh function 
on $\cc^N$ is polynomially convex. 

\medskip
 
 In \cite{Bers1} we showed that $\tbif:= dd^c\chi$ is a {\em bifurcation current}, in the sense that its support is 
 precisely the set of projective structures whose holonomy representation is not locally structurally stable in $P(X)$. 
 Equivalently, the complement $\supp(\tbif)^c$  is 
  the interior of the set of projective structures with discrete holonomy $P_{D}(X)$.  
  A theorem due to 
  Shiga and Tanigawa \cite{shiga tanigawa} and Matsuzaki \cite{matsuzaki} asserts that 
  $\mathrm{Int} (P_D(X)) = P_{QF}(X)$, the set of projective structures with quasifuchsian holonomy, so we conclude that 
  $\supp(\tbif) = (P_{QF}(X))^c$. 
 
Analogous bifurcation currents have been defined for families of rational mappings on $\pu$. It turns out that the exterior powers $\tbif^k$ 
are interesting and rather well understood objects   in that context (see \cite{survey} for an account).
 In particular, in the space of polynomials of degree $d$, the maximal exterior power $\tbif^{d-1}$ is a positive 
measure supported on the boundary of the connectedness locus,  which is the right
analogue in higher degree  of the harmonic measure of the Mandelbrot set \cite{preper}. 

For bifurcation currents associated to spaces of representations, nothing is known in general about the exterior powers $\tbif^k$. 
In our situation, we are able to obtain some information.

\begin{theo}\label{theo:tbif}
Let $X$ be  a compact   Riemann surface of genus $g\geq 2$. 
Let $T_{\mathrm{bif}} = dd^c\chi$ be the natural bifurcation current on $P(X)$. Then $\fr B(X)$ is contained in $\supp(\tbif^{3g-3})$. 
\end{theo}

Notice that $3g-3$ is the maximum possible exponent. 
It is likely that the support of $\tbif^{3g-3}$ is much larger than $\fr B(X)$. 
The reason for the compactness assumption   here is that the proof   is based on results of Otal \cite{otal} and 
Hejhal \cite{hejhal schottky} that are known to hold only when $X$ is compact. 

\medskip 

If $\gamma$ is a geodesic on $X$, we let $Z(\gamma)$ be the subvariety of $P(X)$ defined by the property that 
$\tr^2(\mathsf{hol}(\gamma)) = 4$ (i.e.   $\mathsf{hol}(\gamma)$ is parabolic or the identity).
As a consequence of Theorem \ref{theo:tbif} and of the equidistribution theorems of
 \cite{Bers1} we obtain the following result, which 
 contrasts with the description of $\fr B(X)$ ``from the inside"  in terms of maximal cusps and 
 ending laminations (\cite{minsky, bcm}, see also \cite{lecuire} for a nice account).

\begin{coro}\label{coro:deterministic}
For every $\e>0$ there exist $3g-3$ closed geodesics $\gamma_1, \ldots , \gamma_{3g-3}$ on $X$ such that $\fr B(X)$ is contained in the $\e$-neighborhood of $Z(\gamma_1)\cap \cdots \cap Z(\gamma_{3g-3})$. 
\end{coro}

We observe that the value 4 for the squared trace is irrelevant  here. As the proof will show, the result 
holds a.s. when $\gamma_1, \cdots, \gamma_k$ are independent random closed geodesics of length tending to infinity.

\subsection{Notation}~

$\pu = \mathbb P^1(\mathbb C)$ is the Riemann sphere. $z$  often denotes the variable in $\mathbb P^1$


$\mathbb H$ the upper half plane. $\tau$ a variable in $\mathbb H$

$X$ the finite type Riemann surface on which is defined a projective structure 


$\sigma$ a parabolic projective structure on $X$.

$\mathsf{dev}  = \mathsf{dev}_\sigma: \widetilde{X} \rightarrow \mathbb P^1$ a developing map 

$\hol =  \mathsf{hol}_\sigma: \pi_1(X) \rightarrow \text{PSL} (2,\mathbb C)$ the holonomy representation 

$\pi : M_\sigma \rightarrow X$ the flat $\mathbb P^1$-bundle over $X$ 

$\mathbb P^1 _x = \pi^{-1}(x)$ the  fiber over $x\in X$.

$\varpi : \widetilde{X} \times \mathbb P^1 \rightarrow M_\sigma$ the quotient map.

$s : X\rightarrow M_\sigma$ a holomorphic section of $\pi$

$\mathcal F  $ the holomorphic foliation defined by the flat connexion

$T$ the harmonic current on $M_\sigma$

$\overline M_\sigma$, $\overline{\mathcal F}$, $\overline{s}$, $\overline{T}$, etc.  are the compactifications of the corresponding objects 
when $X$ is not compact. 


\section{Harmonic measures and harmonic currents}

In this section we introduce a number of geometric objects which that will be fundamental in our study: the suspension $M_\sigma$ 
(as well as its compactification $\overline M_\sigma$), and the foliation $\mathcal{F}$
 (resp. $\overline{\mathcal{F}}$). We  also study the ergodic theoretic properties 
  of   $\mathcal{F}$, by way of three closely   related, though slightly different  tools: a family of harmonic measures on the fibers of 
  $M_\sigma$, a foliated harmonic current, and its associated foliated harmonic measure.

\subsection{Generalities}\label{ss:generalities}
We fix once for all a Riemann surface $X$ of finite type, that is, $X$ is biholomorphic to a compact Riemann surface $\overline{X}$ with finitely many points deleted.  We assume that $X$ has negative Euler characteristic. If necessary, we endow $X$ with its hyperbolic metric, which is of finite volume. 

A $\mathbb P^1$-bundle over a Riemann surface $X$ is a holomorphic fibration $M\rightarrow X$ with  $\mathbb P^1$ fibers. It is always the projectivization of a rank $2$ holomorphic vector bundle over $X$. 
If the Riemann surface is compact, the compact complex surface $M$, being a $\mathbb P^1$-bundle over a curve, is algebraic (by 
the GAGA principle), thus in particular it is K\"ahler. We refer to \cite[V.4]{bpvdv}.

A holomorphic $\mathbb P^1$-bundle always admits a smooth section  $X\rightarrow M$  (and even a holomorphic one, see \cite[V.4, p. 139]{bpvdv}). When $X$ is compact, the parity of the self-intersection of such a smooth section   depends only on the fibration; it is even iff $X$ is diffeomorphic to the trivial bundle $X\times \mathbb P^1$, and odd otherwise.

Let $s$ and $f$   respectively denote a smooth section and a fiber of $M\rightarrow X$, then we have that
\begin{equation} \label{eq:homology fiber bundle}  H^2 (M, \mathbb C ) := \mathbb C [s] \oplus \mathbb C [f] , \end{equation}
where $[s]$ and $[f]$ are the cohomology classes dual to $s$ and $f$ respectively. (Throughout this paper, we consistently identify the section $s$ and its graph in  $M$).

In particular, since we can always choose $s$ to be  holomorphic, and  since $M$ is K\"ahler, we obtain an isomorphism between the Dolbeaut cohomology group $H^{1,1}_{\overline{\partial}} (M,\mathbb C)$, the Bott-Chern cohomology group $H^{1,1}_{\fr\overline\fr}(M,\mathbb C)$, and  $H^2 (M,\mathbb C)$. 

\subsection{Parabolic flat $\mathbb P^1$-bundles}\label{ss:model}
Given a parabolic projective structure $\sigma$ on $X$, 
we introduce the flat $\mathbb P^1$-bundle $\mathbb P^1 \rightarrow M_{\sigma}\stackrel{\pi}{\rightarrow} X$ with monodromy 
$\mathsf{hol}_\sigma$, namely the quotient of the flat bundle $\widetilde{X} \times \mathbb P^1$ 
under the action of $\pi_1(X)$ given by 
\begin{equation} \label{eq:action of Gamma} \gamma (x,z) = (\gamma x, \mathsf{hol}_\sigma (\gamma)z).\end{equation}
We denote by $\varpi : \widetilde{X} \times \mathbb P^1 \rightarrow M_{\sigma}$ the natural projection.
Also, we let $\mathcal F$  be the holomorphic foliation on $M_\sigma$ obtained by taking   the quotient of the  horizontal 
fibration $\widetilde{X}\times \set{z}$ of $\widetilde{X} \times \mathbb P^1$.

When $X$ is not compact, 
we can compactify the bundle $M_\sigma$ as a bundle   over  $\overline{X}$.  The 
flat connexion $\nabla$ extends as a meromorphic connexion $\overline{\nabla}$ on $\overline M_\sigma $, 
and the foliation $\mathcal F$ extends as a singular holomorphic foliation $\overline{\mathcal F}$. 

Here are the details. 
Consider the following model for a 
 $\mathbb P^1$-bundle over the unit  disk equipped with a meromorphic flat connexion having a pole over  $0$, 
defined by the differential equation
\begin{equation}\label{eq:model}  \frac{ dv}{du} = \frac{i}{2\pi u} \end{equation}
in coordinates $(u,v)\in \mathbb D\times \mathbb C$. We denote the induced foliation on $\mathbb D \times \mathbb P^1$ by $\mathcal F_m$. The monodromy of the connexion around $u=0$ is the parabolic map $v\mapsto v+1$. Since by assumption the holonomy representation is 
parabolic, we can glue this local model to each of the cusps of $ M_\sigma$ to obtain the desired  $\mathbb P^1$-bundle 
$\overline{M}_{\sigma}$ over $\overline{X}$ equipped with a meromorphic flat connexion $\overline{\nabla}$ and singular holomorphic foliation $\overline{\mathcal F}$. 

\begin{figure}[h]
\centering \vspace{.5cm}
\def\svgwidth{8cm}
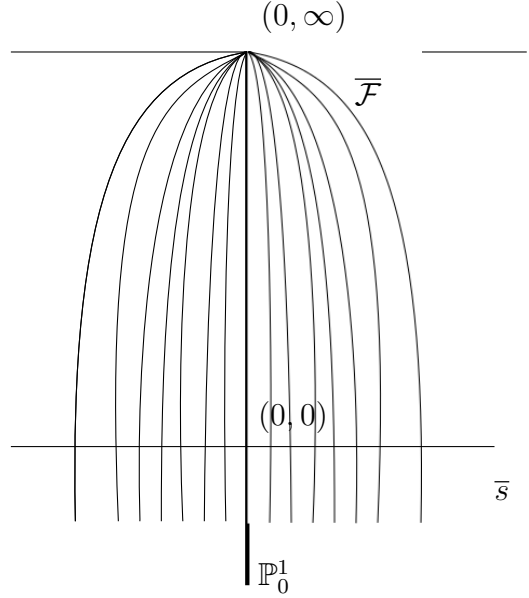
 \caption{Schematic view of the compactified foliation}
  \label{fig:foliation}
\end{figure}

\subsection{Parabolic $\mathbb P^1$-structures and holomorphic sections of flat $\pu$-bundles}

Let now $\mathsf{dev}: \widetilde{X} \rightarrow \mathbb P^1$ be a developing map of the  parabolic $\mathbb P^1$-structure $\sigma$.
The map $\widetilde X\ni x\mapsto (x, \mathsf{dev}(x))\in \widetilde X\times \pu$
 is $\pi_1(X)$-equivariant, hence it descends to  a section $s : X\rightarrow M_\sigma$ 
of the bundle $M_\sigma \rightarrow X$. This section will play an important role in what follows. 

\begin{lem} \label{l:section}
The section $s$ extends to a section $\overline{s} : \overline{X} \rightarrow \overline M_\sigma$ which is transverse to the foliation $\overline{\mathcal F}$. 
\end{lem}

\begin{proof} To make the compactification more explicit, consider  a neighborhood  $N$ of  a puncture 
equipped with a coordinate $x$ 
in which  the projective structure is defined in the punctured unit disk
 by its developing  map $x\mapsto \log x$ 
(see \S\ref{ss:parabolic type PS}). Write $x= \exp (2i\pi \tau)$ with $\tau\in \hh$. 
The $\mathbb P^1$-bundle $M_{\sigma }$ over the punctured disk $  N \simeq \mathbb D ^*$ is the quotient of $\mathbb H \times \mathbb P^1$ by the cyclic group generated by $(\tau, z)\mapsto (\tau + 1, z + 1)$. The map 
\begin{equation} \label{eq:identification} (\tau, z) \longmapsto (u= \exp(2i\pi \tau) , v= z-\tau )\end{equation}
is invariant under this transformation, and it maps the horizontal foliation to the foliation defined by $dz= dv + d\tau = \big(dv + \frac{du}{2i\pi u}\big) = 0$. Hence  \eqref{eq:identification} provides  the identification between the bundle $ M_\sigma$ over $  N \simeq \mathbb D^*$ and the model~\eqref{eq:model}.

The section $s$ of $ M_\sigma$ is defined in the coordinates $(\tau, z)$ to be  the diagonal $z \mapsto (z,z)$, 
so in the coordinates $(u,v)$ it is given by $u \mapsto (u,0)$. 
Hence the section $s$ extends as a section $\overline{s}$ of $\overline{ M_\sigma}$. 
\end{proof}

\subsection{Fiberwise harmonic measures}\label{ss:fiberwise}

Fix an biholomorphism between  $\widetilde{X}$ and $\mathbb H$, 
 thereby identifying $\pi_1(X)$ with a lattice $\Gamma$ in $\PSLR$. 
Recall from \cite{furstenberg, margulis} that if $\rho: \Gamma\cv\PSL$ is a non elementary representation,
 there exists a unique (Lebesgue) measurable $\rho$-equivariant map 
 $\Phi : \mathbb P^1(\mathbb R) \rightarrow \mathbb P^1$ defined almost everywhere. Likewise, if $\text{Prob}(\mathbb P^1)$ 
 denotes the compact convex set of probability measures on $\mathbb P^1$ (endowed with the weak* topology), then the map
 $a \in \mathbb P^1(\mathbb R) \mapsto 
 \delta_{\Phi(a)}  \in \text{Prob} (\mathbb P^1)$ is the unique measurable  $\rho$-equivariant map.

 A measurable family of probability measures $(m_{\tau})_{\tau\in \hh}$ 
 on $\pu$ is said to be {\em harmonic} if for every test function $\psi$, the function 
 $\hh\ni \tau\mapsto \int  \psi m_\tau $ is harmonic.

\begin{prop}\label{p:harmonic measures}
 Let $\Gamma$ be a lattice in $\PSLR$ and  $\rho: \Gamma\cv\PSL$ be a non elementary representation. 
 Then there exists a unique measurable family of probability measures $\{ \nu_\tau \} _{\tau\in \mathbb H} $ on $\pu$ such that :
\begin{enumerate}[\em (i)] 
\item $\tau \mapsto \nu_\tau$ is harmonic; 
\item $\nu_{\gamma \tau } = \rho(\gamma)_* \nu_\tau$ for every $\gamma \in \pi_1(X)$ and every $\tau\in \mathbb H$. 
\end{enumerate}  
In particular,  we have the formula 
\begin{equation}\label{eq:harmonic current} 
\text{ for every } \tau \in \mathbb H,\  \nu_\tau  = \frac{1}{\pi}  \Phi _*  \big( \frac{\Im \tau \ d a}{ |\tau - a |^2} \big),  
\end{equation}
and these measures coincide with those defined in Definition \ref{defprop:harmonic measure}.
\end{prop}

The family $\set{\nu_\tau, \tau\in \hh}$  will be simply referred to as the {\em family of harmonic measures} associated to the $\pu$-structure. 

\begin{proof} 
It will be  convenient to view $\mathbb P^1(\rr)$ as the boundary of the upper half plane. 
First, it follows from the Poisson formula and the equivariance of $\Phi$  that \eqref{eq:harmonic current} 
defines a family of harmonic measures $\nu_\tau$ satisfying (i) and (ii). 

To establish the uniqueness statement, fix  
 a family of probability measures  $\{\nu'_\tau\}_{\tau \in \hh}$  on $\mathbb P^1$ satisfying (i) and (ii). Then, 
  Fatou's 
theorem on boundary values on bounded harmonic functions implies that   there is a measurable map $\widehat \nu' : \partial \mathbb H \rightarrow \text{Prob} (\mathbb P^1)$ with values in the set of probability measures on $\mathbb P^1$, such that the Poisson formula holds, namely for every $\tau \in \mathbb H$
\[  \nu'_\tau = \frac{1}{\pi}\int_{\mathbb R}  \frac{\Im \tau }{ |\tau - a |^2} \widehat \nu' (a) da .\]
The map $\widehat \nu'$ is $\rho$-equivariant since   the family $\{\nu'_\tau\}_{\tau \in \mathbb H}$ is. Hence 
by the observations preceding the proposition, we get 
 that a.s. $\widehat{\nu'}(a)$ is the Dirac mass at $\Phi (a)$, and we are done. 
%

The fact that these measures coincide with the ones from Definition-Proposition
 \ref{defprop:harmonic measure} follows from this uniqueness. 
Indeed, the family of measures defined by Definition \ref{defprop:harmonic measure} 
clearly satisfies the equivariance property (ii).  
To check (i) we adapt the classical argument for the harmonic dependence of the harmonic measures with respect to the starting point. 
Indeed let $\mathrm e (\omega)$ be the endpoint mapping  defined  in Definition \ref{defprop:harmonic measure}. 
Let $B\subset\pu$ be any Borel set. Let us prove that $u:x\mapsto \pp_x(e(\omega)\in B)$ is harmonic. 
For this, identify $ \widetilde X$ with the unit disk and $D$  be a small disk centered at $x\in \widetilde X$. 
For $\omega\in \Omega_x$, let $T = \inf\set{t>0, \omega(t)\notin D} $.   
It follows from the strong Markov property of Brownian motion in $\widetilde X$, that 
$$u(x)  = \ee_x\lrpar{\pp_{{\omega(T)}}(\mathrm e(\omega(\cdot - T)) \in B} = \ee_x\big(u(\omega(T))\big) = \int_{\fr D} u.$$ 
Therefore $u$ satisfies the mean value property and the result follows.
\end{proof}

\begin{rmk}\label{r:measurable isomorphism} 
The map $\Phi$ is in  general not injective on any full measure subset of $\mathbb P^1(\mathbb R)$. However 
a theorem of Ledrappier shows that 
this is the case when $\rho$ is faithful and discrete, see \cite{ledrappier poisson boundary}. 
This result will be used  in section \ref{s:dimension}.
\end{rmk}

\subsection{Harmonic currents} \label{ss:harmonic currents} 
Given a foliated complex surface $(M,\mathcal F)$ (possibly with  singularities), 
a {\em directed} (or {\em foliated}) {\em  harmonic current} (often simply abbreviated as   ``harmonic current" in the sequel) is a positive current of bidegree $(1,1)$
which is $\partial \overline{\partial}$-closed, and such that 
$\langle T, \psi \rangle\geq 0$ if $\psi$ is a (1,1) form which is positive along the leaves.  

Such currents have the following local structure outside the singular set $\text{sing} (\mathcal F)$: 
in a foliation box biholomorphic to the bidisk $\mathbb D \times \mathbb D$ 
in which  the foliation is the horizontal fibration, there exists a finite 
positive measure $m$ on $\mathbb D$, and a 
non-negative bounded measurable function 
such that 
\begin{equation}\label{eq:foliated}
 T = \int \varphi \ [\mathbb D \times \set{w}] \ dm(w) ,
 \end{equation}
and moreover $\varphi (.,w)$ is harmonic for $m$-a.e. $w\in \mathbb D$. Here as usual  $[\mathbb D \times \set{w}]$ stands for
 the current of integration on $\mathbb D \times \set{w} $.
The product $\varphi m$ is a well-defined object, which can be thought of as a transverse measure for the foliation 
$\mathcal F$. In particular, if $C\subset M$ is a holomorphic curve disjoint from $\text{sing}(\mathcal{F})$, 
we can define  the restriction $T\rest{C}$ of $T$ to $C$, 
also referred to 
as the geometric intersection $T\geom [C]$ between $T$ and $C$). Observe that the same
    makes sense for any current of the form \eqref{eq:foliated}, whenever harmonic or not.

The existence of a harmonic current directed by the foliation is classical when $M$ is compact and $\mathcal{F}$ is non-singular 
(see e.g. \cite{Ghys}); the singular case was treated in  \cite{BS}. 

\medskip

Assume now that $\mathcal{F}$ is the foliation by flat sections of $M_\sigma$. 
There is a 1-1 correspondence between foliated harmonic currents and the fiberwise harmonic measures of \S \ref{ss:fiberwise}. 
Indeed, consider a foliated harmonic current $T$ on $M_\sigma$, normalized so that one (hence all) of its vertical slices is of unit mass. 
Lifting $T$ to the universal cover $\hh\times \pu$, 
we obtain a   harmonic current $\widetilde{T}$
 directed by the horizontal fibration,  that is invariant with respect to the 
 action of $\Gamma\simeq \pi_1(X)$ defined in \eqref{eq:action of Gamma}. 
Restricting to the vertical fibers $\set{\tau}\times \pu$ we get 
 a family of measures $\nu_\tau$ 
 which is easily seen to satisfy the assumptions of Proposition 
 \ref{p:harmonic measures}. 
 
 Conversely, any family of measures  $(\nu_\tau)_{\tau\in \hh}$   on $\set{\tau}\times \pu$ 
 satisfying the assumptions of Proposition \ref{p:harmonic measures} gives rise to a foliated harmonic current on $M_\sigma$. 
 For this, working  first on $\hh\times \pu$, we   construct from (i) a harmonic current $\widetilde T$ directed by the horizontal fibration. 
 Indeed,  the Poisson formula asserts that $\nu_\tau$ is a convex combination  of measures of the 
 form $h(\tau, a)\delta_{\Phi(a)}$, where $\tau\mapsto h(\tau, a)$ is harmonic. Then we get $\widetilde T$ by taking 
  the corresponding combination of currents of the form $h(\tau, a)[\hh\times \set{\Phi(a)}]$. 
  From the equivariance property (ii), $\widetilde{T}$ descends to a foliated harmonic current on $M_\sigma$ and we are done.  


 The following uniqueness statement will be of utmost importance to us. When $X$ is compact it was already established in \cite{dk}.
 
 \begin{prop}\label{p:extension}
Let $X$, $\sigma$ and $M_\sigma$ be as above.
 The singular foliation $\overline{\mathcal{F}}$ on the compactified suspension $\overline{M}_\sigma$ 
 admits a unique normalized foliated harmonic current, carrying no mass on the fibers over the punctures.
 \end{prop} 
 
\begin{proof}
In $M_\sigma$, the existence and   uniqueness of a foliated harmonic current $T$ giving mass 1 to the vertical fibers    follows 
from the above  discussion, together with Proposition \ref{p:harmonic measures}.   Thus, the point is   to show that  $T$ admits an
extension to  a harmonic current  $\overline{T}$ on $\overline{M}_\sigma$ with no mass  on the fibers over the punctures, which is then 
necessarily unique.   

Recall that the foliation $\mathcal{F}$ has a well defined rigid model $\mathcal{F}_m$ in a neighborhood of each puncture, which was defined in  \S\ref{ss:model}. 
%
%
%
%
%
%
The key is the following lemma. 

\begin{lem}\label{l:extension model}
Consider the model foliation 
$ \mathcal F_m$  on $\mathbb D\times \mathbb P^1$, as defined in \S\ref{ss:model}. Let $T$ be any foliated
harmonic current   in $\mathbb D^* \times \mathbb P^1$, normalized so that 
   the restriction of $T$ to any fiber $u\times \mathbb P^1$, 
$u\neq 0$ is a probability measure. Then the restriction of $T$ to the curve $s^* = \{ 0< |u|\leq e^{-2\pi},\ v= 0\}$ has finite mass.
\end{lem}

From this  and Lemma \ref{l:section} (see also Figure \ref{fig:foliation}),
 we deduce that  the harmonic current extends to $\mathbb D \times \mathbb P^1 \setminus \text{sing} (\mathcal F_m)$. It then follows from 
 general extension results for harmonic currents (see e.g. \cite[Thm 5]{dee}) 
 that  it also compactifies at the singular 
 points of $\mathcal F_m$. The proposition follows. \end{proof}
 
\begin{proof}[Proof of Lemma \ref{l:extension model}]
The harmonic current $T$ lifts as a harmonic current $\widetilde{T}$ on $\mathbb H\times \mathbb P^1$ which is defined in the $(\tau,z)$-coordinates by a family of measures $\{  \nu_\tau \}_{\tau \in \mathbb H}$ satisfying 
\begin{equation} \label{eq:equivariance} \nu_{\tau +1} = (z+1)_* \nu_\tau ,\end{equation}
and   depending harmonically on $\tau$.  
As in the proof of Proposition \ref{p:harmonic measures}, the Poisson formula implies  that there exists a family of probability measures $\{ \nu_a \}_{a\in \mathbb R}$ defined for a.e. $a\in \mathbb R$ and depending measurably on $a$, such that for every $\tau\in \mathbb H$
\[  \nu_\tau = \int _{\mathbb R}   \frac{\Im \tau }{(\Re \tau - a)^2 + \Im \tau ^2} \nu_a da.\]
The equivariance relation~\eqref{eq:equivariance} implies that 
\[ \nu_{a+1} =  (z+1)_* \nu_a  \]
almost everywhere. A canonical example of such a family of measures is given by $\nu^{\rm can} _a = \delta_a$ the Dirac mass at the point 
$a$. It defines a harmonic current $T^{can}$ (corresponding to the harmonic current on the suspension corresponding 
to the identity representation). 

A fundamental domain for the pull-back of $s^*$ in $\mathbb H \times \mathbb P ^1$ is the subset $D\times D$
 of the diagonal in
$ \hh\times \pu$, 
  where 
$$D = \set{ \frac{-1}{2}\leq \Re \tau \leq \frac{1}{2} , \ \Im \tau \geq 1 }\subset \mathbb H.$$ 
Therefore, we need to prove that the   integral 
\[  I = \int _{\mathbb R}  da \cdot   \int _D \frac{\Im \tau }{(\Re \tau - a) ^2 + (\Im \tau )^2} \nu_a(d\tau)  \]
is finite. Performing the change of variable $a= b+n$ yields
\[  I =\int _0 ^1 db \cdot  \sum _{n\in \mathbb Z}    \int _D \frac{\Im \tau }{(\Re \tau - (b+n)) ^2 + (\Im \tau )^2} \nu_{b+n}(d\tau) .\]
The equivariance relation $(z+n)_* \nu_b = \nu_{b+n}$ gives 
\[  \int _D \frac{\Im \tau }{(\Re \tau - (b+n)) ^2 + (\Im \tau )^2} (z+n)_* \nu_b(d\tau)  = \int _{D-n}  \frac{\Im \tau }{(\Re \tau - a) ^2 + (\Im \tau )^2} \nu_a(d\tau) \]
where $D-n = \{ \tau -n \ |\ \tau \in D\}$, and we conclude that
\[  I = \int_0 ^1 da \cdot \int_{\Im \tau \geq 1} \frac{\Im \tau }{(\Re \tau - a) ^2 + (\Im \tau )^2} \nu_a(d\tau)  \leq 1 \]
since $\nu_a$ is a probability measure on $\mathbb P^1$ and $\ \frac{\Im \tau }{(\Re \tau - a) ^2 + (\Im \tau )^2} \leq 1$ if $\Im \tau \geq 1$. The proof is complete.
\end{proof}

\begin{rmk}\label{r:measurable conjugacy}
Identify $\widetilde{X}$ with $\mathbb H$ via a biholomorphism. For any $\mathbb P^1$-structure $\sigma$, the map $\Phi$ can be used to construct a measurable map $M_{\sigma_{ \rm Fuchs}} \rightarrow M_\sigma$   mapping biholomorphically every leaf of $\mathcal F_{\sigma_{ \rm Fuchs}}$ to a leaf of $\mathcal F_\sigma$. At the level of the universal covers, this map is simply given by  $(\tau, z) \mapsto (\tau , \Phi(z))$. Observe furthermore, that the normalized current $T_{\sigma_{ \rm Fuchs}}$ is mapped 
 to  $T_\sigma$ (indeed, this holds for the fiber harmonic measures). If in addition 
  the holonomy is faithful with discrete image, Remark \ref{r:measurable isomorphism} shows that the foliations $(M_{\sigma_{ \rm Fuchs}}, \mathcal F_{\sigma_{ \rm Fuchs}}, T_{\sigma_{ \rm Fuchs}})$ and $(M_\sigma, \mathcal F_\sigma, T_\sigma)$  are actually measurably conjugated. 
\end{rmk}

\subsection{Foliated harmonic measures: Garnett's theory} \label{ss:garnett} 
In this paragraph we briefly review Garnett's theory of foliated Brownian motion \cite{garnett} (see also \cite{candel}), and adapt it to our non compact situation.
Let us define the normalized measure  
\begin{equation} \label{eq:harmonic measure} \mu = \frac{1}{\text{vol} (X) } \text{vol}_P \wedge T  .\end{equation}
This measure is a \textit{harmonic measure in the sense of Garnett}, namely
 it satisfies the equation $\Delta_{\mathcal F} \mu = 0$ in the weak sense, here $\Delta_{\mathcal F}$ 
 is the leafwise laplacian relative to the leafwise Poincar\'e metric. (This is immediate from the fact that $T$ itself is harmonic.) We will refer to such measure as \textit{foliated harmonic measures}.

Let $\Pi = \{ \Pi_t \}_{t\geq 0}$ be the Markov semi-group of operators acting on $C^0 _c (M_\sigma)$, 
  whose infinitesimal generator is $\Delta_{\mathcal F}$. It is convenient to  consider it at the level of the universal cover 
$\widetilde{X} \times \mathbb P^1$.  There it expresses as 
\begin{equation}\label{eq:markov operator} \Pi_t f(x,z) = \int _{\widetilde{X}} p(x,y,t) f(y,z) \text{vol}(dy) \end{equation}
 where $p(x,y,t)$ is the fundamental solution  of the heat equation $\frac{\partial}{\partial t} =  \Delta_{\rm Poin}$ 
 on the hyperbolic plane. Then, since $\mu$ satisfies $\Delta_{\mathcal F} \mu = 0$, it is invariant under the semi-group 
 $\Pi$. 

The following is essentially a reformulation of Proposition \ref{p:extension}. 
 The proof will be left to the reader.

\begin{prop} \label{p:ergodic}
The measure $\mu$ is the only normalized foliated harmonic measure in the sense of Garnett for $\mathcal{F}$ on $M_\sigma$. 
In particular   any measurable subset of $M_\sigma$ which is saturated by $\mathcal F$ has zero or full $\mu$-measure.
\end{prop} 


Consider the Markov process on $ M_\sigma$ induced by the leafwise Brownian motion, 
with respect to the Poincar\'e metric  (recall that the Brownian motion is generated by the operator $ \Delta$). More 
precisely,  we let $\Omega^\mathcal{F}$ be the set of semi-infinite continuous paths 
$\omega : [0,\infty ) \rightarrow M_\sigma$ which are contained in a leaf of $\mathcal F$, and  
$\sigma = \{ \sigma_t \}_{t\in [0,+\infty)}$ be the shift semi-group acting on $\Omega$ by $\sigma_t (\omega) (s) = \omega (t+s)$. 
Let 
\[ W^{\mathcal F}_\mu := \int W^{\mathcal F}_x \ d\mu (x) \]
on $\Omega^\mathcal{F}$, where $W^{\mathcal F}_{x}$ is the Wiener measure on the subset $\Omega^{\mathcal F}_x$ of paths starting at $x$. We also sometimes use the Wiener space $(\Omega^X, W^X)$ of Brownian paths on $X$. 

 The following proposition is contained in \cite[\S 6]{candel}. For the sake of convenience we sketch the proof. 

\begin{prop} 
The measure $ W^{\mathcal F}_\mu$ is $\sigma$-invariant and the dynamical 
 system $(\Omega^\mathcal{F}, \sigma, W^{\mathcal F}_\mu)$ is ergodic. 
\end{prop}

\begin{proof} 
Let us first  show that $W^{\mathcal F}_{\mu}$ is $\sigma$-invariant. 
Let $E \subset \Omega$ be a measurable subset. By the Markov property, for every $x\in M_\sigma$ and every $t\geq 0$ we have that 
\begin{equation} \label{eq:invariance} W^{\mathcal F}_x (\sigma_{t}^{-1} E) = \int_{L_x} p(x,y,t) W^{\mathcal F}_y ( E) d y . \end{equation}
Consider the   function $f:(t, x) \mapsto W^{\mathcal F}_x (\sigma_{t}^{-1} E)$. Equation \eqref{eq:invariance} shows that $f$ satisfies the heat equation, with initial condition $f(0, x) = W^{\mathcal F}_x (E)$, hence for every $t\geq 0$,  
 $f(t,.) = \Pi_t f(0,.)$. 
Since $\mu$ is invariant under the heat semi-group, we deduce that 
\[  W^{\mathcal F}_\mu (\sigma^{-t} (E))=  \int _{X_\rho} f(t,x) d\mu(x) = \int _{X_\rho} f(0,x) d\mu(x) = W^{\mathcal F}_\mu (E), \]
hence proving the first part of the proposition. 

\medskip

We now prove that  $(\Omega^\mathcal{F}, \sigma, W^{\mathcal F}_\mu)$ is ergodic. Let $E$ be any $\sigma$-invariant subset. 
The function $x\mapsto f(0,x) = W^{\mathcal F}_x(E)$ is then measurable, bounded, and harmonic along $\mu$-a.e. leaf. We claim
 that it is constant. Indeed, observe that 
  for any $c\in \mathbb Q$, the function $g= \max (f,c)$ is leafwise subharmonic on a.e. leaf, so  
  we get that for every $t\geq 0$,  $\Pi_t g \geq g$ on a.e. leaf. On the other hand, $\int \Pi_t g \; d\mu = \int g \; d\mu$, 
  so we infer  that
  on a set of full measure  $\Pi_t g = g$ holds   for every rational $t\geq 0$.  This proves that $g$ is harmonic on $\mu$-a.e. leaf. This being true for every $c$, it follows that $f$ is constant  along a.e. leaf. 
Now,   since $E$ is shift invariant, belonging to $E$ is a a tail property, so
  by applying the  0-1 law \cite[Prop. 6.5]{candel}  we infer that $E$ has zero or full measure on a.e. leaf.
 Applying Proposition \ref{p:ergodic} then concludes the proof. 
 \end{proof}

\section{The degree} \label{s:degree}

In this section we introduce 
 the concept  of the degree of  a $\pu$-structure on $X$. 
We justify its existence in \S\ref{ss:existence of degree} by  
  proving  Proposition~\ref{defprop:degree}.  Then in \S \ref{ss:vanishing}, 
we characterize projective structures  with vanishing degree, and in \S\ref{ss:cohomological degree}
we show that it can be expressed in terms of   cohomological data.

\subsection{Existence of the degree and equidistribution of large leafwise discs}\label{ss:existence of degree} 
This subsection is devoted to the proof of  Definition-Proposition~\ref{defprop:degree}. 
Recall that we are given a developing map $\mathsf{dev}:\widetilde X\cv \pu$  
 of a parabolic $\mathbb P^1$-structure with non-elementary holonomy, and wish to show that 
 $\unsur{\mathrm{vol} (B(x_n, R_n))} \#\set{B(x_n, R_n)\cap \mathsf{dev}^{-1}(z_n)}$ 
 converges to some limit $\delta$, independent of the choices. 
 To ease notation we set  $\mathrm{vol}(R_n) = \mathrm{vol} (B(x_n, R_n))$. 
 Using the equivariance we may  assume without loss of generality that $(x_n)$ is relatively compact in $\widetilde X$. 
 Recall that the graph 
 of the developing map in $\widetilde X\times \pu$
descends to a section $s$ of $M_\sigma$ transverse to $\mathcal{F}$.
 Recast in geometric language, we need to show that 
$\unsur{\mathrm{vol}(R_n)}\#(B(x_n, R_n))\times\set{z_n} ) \cap \mathrm{Graph}(\mathsf{dev})$ converges to some value $\delta$. 
  Pushing forward by $\varpi$, this amounts to proving that the geometric intersection number  
$$\int_{M_\sigma} \varpi_*\lrpar{\frac{1}{\mathrm{vol}(R_n)} [ B(x_n,R_n)\times\set{z_n} ]}\geom [s]$$ 
converges to $\delta$.
Put $\Delta(R_n) =\varpi_*\lrpar{ \frac{1}{\mathrm{vol}(R_n)} [ B(x_n,R_n)\times\set{z_n} ]}$, which is 
 a current with boundary supported in a  leaf of $\mathcal F$.
 It is perhaps useful to  stress here that $\Delta(R_n)$, may be decomposed into pieces of varying  multiplicities (according to the self-
 overlapping properties of $\varpi(B(x_n, R_n))$), and that these multiplicities are taken into account in the geometric 
 wedge product $\geom$. 
   
\medskip

The key    is the following   equidistribution result for  large leafwise discs in parabolic flat $\mathbb P^1$-bundles, which is 
 due to  Bonatti and G\'omez-Mont  \cite{bgm}, given the positivity of the foliated   Lyapunov exponent,
 a fact that was established in this generality in our previous work \cite{Bers1}. 

\begin{prop} \label{p:bgm}
Let $\rho : \pi_1 (X) \rightarrow \text{PSL} (2,\mathbb C)$ be a non elementary representation. 
 Let $(x_n)_{n\geq 0}$ 
  be a sequence in $\widetilde X$ such that $(c(x_n))_{n\cv\infty}$  is relatively compact in $X$. Let $(R_n)$ be 
  a sequence of positive real numbers tending to $+\infty$, and $(z_n)$ be any sequence of points on the Riemann sphere. 
  Then the projection in $M_\sigma$ of the sequence of integration 
   currents $\Delta(R_n)  = \varpi_*\lrpar{ \frac{1}{\mathrm{vol}(R_n)} \left[ B(x_n,R_n) \times\set{ z_n}\right]}$ converges  to $\unsur{\vol(X)}T$ 
   when $n$ tends to infinity. 
\end{prop}

\begin{proof} 
Let $\text{vol}$ denote  the Poincar\'e volume form along the leaves of $\mathcal F$.   Remark  that since all currents are directed by the foliation,  the convergence   $ { \Delta(R_n)}\underset{n\cv\infty}{\longrightarrow} \frac{1}{\mathrm{vol} (X)} T $ 
 is equivalent to that of $\varpi_*\lrpar{\frac{1}{\mathrm{vol}(R_n)} \text{vol} \rest{B(x_n,R_n)\times\set{z_n} }} $ towards the measure 
$\mu :=\frac{1}{\mathrm{vol} (X)} T\wedge \text{vol}$. 

By  \cite[Thm 2]{bgm},  for this it is enough to show that the top Lyapunov exponent of the cocycle induced by $\rho$ over the 
geodesic flow on $T^1X$ is positive. The representation  $\rho$ being non elementary, 
this positivity was shown in \cite[Rmk 2.19]{Bers1}. The result follows. 
\end{proof}

 We see that to prove the desired result, it is enough to show that 
\begin{equation}\label{eq:counting}
\int_{M_\sigma}  {\Delta(R_n)}\geom [s] \underset{n\cv\infty}{\longrightarrow} \frac{1}{\mathrm{vol} (X)} 
 \int_{M_\sigma} T\geom [s].
\end{equation} 
     We note 
    that it follows from the previous proposition that if $\alpha$ is any smooth form along the leaves of $\mathcal{F}$, 
 $\bra{\Delta(R_n), \alpha}$ converges to $\bra{\unsur{\vol(X)} T, \alpha}$. 
  The proof of \eqref{eq:counting}
   will be carried out in several steps. As it is common in such counting issues, 
   special attention must be paid to boundary effects.

\medskip

\textit{Step 1.} Here we prove \eqref{eq:counting} on compact subsets of $M_\sigma$.  Since $s$ is a section of $M_\sigma \cv X$, it is 
enough to test the convergence on  test functions  of the form $\pi^* \psi$, with $ \psi\in \mathcal{C}_c(X)$, 
which  we simply denote by $ \psi$, that is, we need to show that 
$\bra{\Delta(R_n)\geom [s],  \psi} \cv \frac{1}{\mathrm{vol} (X)} \bra{T\rest{s},  \psi}$. Fix $\e>0$.  
 To lighten notation, we put $B_n(R_n) = B(x_n, R_n)\times \set{z_n}$.


We first construct a regularization of  $ \psi[s]$. Since $s$ is transverse to $\mathcal{F}$ we can  extend 
$ \psi$ locally around $s$ to be constant along the leaves. Fix a non-negative smooth function  
$\theta_{\varepsilon} : [0,\infty) \rightarrow [0,\infty)$  with support contained in $[0,\varepsilon]$, and 
such that $\int_{\mathbb D} \theta_{\varepsilon} (d_{\rm Poin}(0,x)) \vol(dx) = 1$.
Let now $\Delta$ be a foliated current, expressed  as $\Delta  = \int \varphi [\dd\times \set{w}] dm(w)$
in a flow box   around a point of $s$, in which $s$ corresponds to $\set{0}\times \dd$. 
Define a form along the leaves by 
$$( \psi[s])_\e = \theta_\e(d_{\mathcal{F}}(\cdot, s))\psi \vol_\mathcal{F},$$ where  $d_{\mathcal{F}}$ (resp. $\vol_{\mathcal{F}}$) is the leafwise Poincar\'e 
distance (resp. volume form). 
If $\varphi$ is continuous, then clearly 
$\Delta\wedge ( \psi[s])_\e $ is close to $\Delta\geom ( \psi[S])$ 
(when $\Delta= \Delta(R_n)$, this will happen when $\fr\Delta(R_n)$ is far from $s$).

We then write 
 \begin{align} \label{eq:Delta}
 & \int  \Delta(R_n)\geom  \psi [s] -  \frac{1}{\mathrm{vol} (X)} \int T\geom  \psi[s] = \lrpar{
\int \Delta(R_n) \wedge  ( \psi [s] )_\e- \frac{1}{\mathrm{vol} (X)} T\wedge  ( \psi [s] )_\e }  + \\ \notag
&+ \int \Delta(R_n) \geom  \psi [s] -\Delta(R_n) \wedge ( \psi [s] )_\e 
+\frac{1}{\mathrm{vol} (X)}  \int  T\wedge ( \psi [s] )_\e-T\geom  \psi [s]   
\end{align}
as a sum of three terms $ I+II+III$.  Since $( \psi [s] )_\e$ is smooth along the leaves,  
Proposition \ref{p:bgm} implies that $I$ converges to zero as $n\cv\infty$. 
Since in the representation  \eqref{eq:foliated} the density $\varphi$ of $T$ along the leaves is harmonic, the mean value formula implies that the integral $III$ vanishes.  

We will decompose the integral $II$ as a sum of two contributions.  We declare that  a  point in 
$B_n(R_n)\cap \varpi^{-1}(s)$ is a  good 
intersection if the 
 ball $B_{Poin}(p,\e)$ of radius $\e$ relative to the Poincar\'e metric is  disjoint from $\fr (B_n( R_n))$. 
Therefore, $ B_n( R_n)\cap \varpi^{-1}\supp( \psi[s])_\e$  is a union of good and bad components. Notice that bad components are 
contained in a leafwise $2\e$-neighborhood of $ \fr B_n( R_n)$. Pushing forward again by $\varpi$ we 
let $\Delta_n^{\rm bad}$ be the part of $\Delta(R_n)$ corresponding 
to bad components and $\Delta_n^{\rm good}$  be its complement (notice that $\Delta_n^{\rm good}$ is larger than the union of good 
components). Since $ \psi$ is constant along the leaves near $s$, by definition of $( \psi[s])_\e$ we get that 
$$   \int \Delta_n^{\rm good}\geom ( \psi[s]) =  \int \Delta_n^{\rm good}\wedge   ( \psi[s])_\e) .$$ To estimate the contribution of the bad part, observe that bad components of $\Delta(R_n)$ become good in $\varpi (B_n( R_n+2\e))$ as well as in the annulus 
$\varpi (B_n( R_n+2\e)\setminus B_n( R_n-2\e))$. So  we infer that 
\begin{align*}
\int \Delta_n^{\rm bad} \geom  \psi [s] &
\leq \int \unsur{\mathrm{vol}(R_n)}\lrpar{\varpi_* \left[B_n( R_n+2\e)\setminus B_n( R_n-2\e)\right]}^{\rm good} \geom  \psi[s] \\
 &=  \int \unsur{\mathrm{vol}(R_n)}\lrpar{\varpi_* \left[B_n( R_n+2\e)\setminus B_n( R_n-2\e)\right]}^{\rm good} \wedge( \psi[s])_\e \\
 &\leq  \int \unsur{\mathrm{vol}(R_n)} {\varpi_* \left[B_n( R_n+2\e)\setminus B_n( R_n-2\e)\right]} \wedge ( \psi[s])_\e\\
&\underset{n\cv\infty}{\longrightarrow} \lrpar{e^{2\e} - e^{-2\e}} \int T  \wedge ( \psi[s])_\e
 = \lrpar{e^{2\e} - e^{-2\e}} \int T \geom  \psi[s] =O(\e)
 \end{align*}
where the convergence in the last line follows from Proposition \ref{p:bgm} and the fact that $ {\mathrm{vol}(R_n+2\e)}\underset{n\cv\infty}\sim e^{2\e}{\mathrm{vol}(R_n)}$.
 We thus   conclude that the difference of integrals in \eqref{eq:Delta} is arbitrary small as $n\cv\infty$, and Step 1 is complete.

\bigskip

\textit{Step 2.} 
To show that the convergence \eqref{eq:counting} holds throughout $M_\sigma$, we work in the compactification 
$\overline M_\sigma$. Let $\mathbb P^1_p$ be the fiber of $\overline M_\sigma\cv\overline{X}$ over a puncture $p$.    
 We know from Lemma \ref{l:extension model} that the measure $\overline T\geom [\overline s]$ has finite mass. 
Since $\overline T$ carries no mass on $\mathbb{P}^1_p$,
from the local picture of $\mathcal{F}$ and $s$ given in \S\ref{ss:model}, we infer that the measure $\overline T\geom [\overline s]$ 
has no atom at $\overline  s(p)$. Therefore, to prove the desired convergence it is enough to show that the mass 
of $\Delta(R_n)\geom [s]$ near $\overline s(p)$ is uniformly small with $n$. 

We   use the local model for $\sigma$ near $p$.  
 Fix a  coordinate $z$ in which 
 $N(p)$ is identified to $\dd^*$ and  the projective structure   is given by $\log z$. Let $N_\e(p)  = \set{0<\abs{z}<\e}$.  
  Then any connected component of $c^{-1}(N_\e(p))$ 
 is the interior of a horocycle in $\hh$. 
 The crucial point is that  the developing map is injective in any component of $c^{-1}(N_\e(p))$. In particular  the cardinality 
 of $\mathsf{dev}^{-1} (z_n) \cap B_n(R_n) \cap c^{-1}(N_\e(p))$ 
 is bounded by the number of connected components   of $c^{-1}(N_\e(p))$ intersecting $B_n(R_n)$.  
 Now we observe that there is a universal constant $\alpha>0$ such that if $U$ is such a component,
  then the area of  $B_n(R_n +1)\cap U$ is at least $\alpha$.
  From this we infer that  
 \begin{equation}
 \label{eq:component}
   \#\set{\mathsf{dev}^{-1}(z_n) \cap B_n(R_n) \cap c^{-1}(N_\e(p))} \leq \unsur{\alpha} 
 \mathrm{vol}_\hh \lrpar{B_n(R_n + 1) \cap c^{-1}(N_\e(p))}.
 \end{equation} 
 It is well known that the image of $B_n(R_n + 1)$ under $c$ becomes asymptotically equidistributed in $X$ as $n\cv\infty$.
 This may be obtained as a consequence of Proposition \ref{p:bgm},  but it already follows from Margulis \cite{margulis these}. 
  From this and \eqref{eq:component}, we conclude that 
 $$\unsur{\mathrm{vol}(R_n)} \#\set{D^{-1}(z_n) \cap B_n(R_n) \cap c^{-1}(N_\e(p))} $$ is bounded by $C \vol_X(N_\e(p))$
 and the result follows.\qed

\begin{rmk}
This proof shows that rather than a simple number, it is more precise to view  the degree 
 as a positive measure on $X$, defined by  the formula
$\deg(\sigma) =   \pi_* (T\geom[s])$. 
In particular it makes sense to speak of the degree of $\dev$ restricted to some $\pi_1(X)$-invariant subset of $\widetilde X$. 
This measure is canonically associated to the $\mathbb P^1$-structure on $X$ (that is it does not depend on the chosen developing map). 

The support of the degree measure can be described as follows. Let $(\dev , \hol)$ be a development-holonomy pair for the structure. Let $
\Lambda \subset \mathbb P^1$ be the limit set of $\hol$. The closed subset $\dev^{-1} (\Lambda)$ being invariant by $\pi_1(X)$, it defines a 
closed subset $\Lambda_\sigma \subset X$. This set is canonically associated to the projective structure and does not depend on the chosen development-holonomy pair.
An instructive example is given by  $\mathbb P^1$-structures with Fuchsian holonomy. In this case the set $\La_\sigma$ 
is a union of the boundaries of disjoint annuli embedded in $X$. It was 
 studied e.g. by Goldman to prove that such structures   are obtained from  $2\pi$-graftings, see \cite{Goldman grafting}. 
 
 We claim that the support of the degree is the 
 set $\Lambda_\sigma$. Indeed, at the level of the universal cover, the pull-back of the degree is the intersection of $\widetilde{T}$ with the 
 graph of $\dev$. In particular
 ence by the Harnack inequality, it is absolutely continuous with respect to the pull-back by $\dev$ of \textit{any} harmonic measure, 
 with density bounded from above and below by positive constants. 
 The claim then follows from the fact that the support of the harmonic measures is the limit set of $\hol$. 
 This argument shows more: namely, that the degree has the same Hausdorff dimension of that of the harmonic measures. 
 In particular, our Theorem \ref{theo:dimension} shows that the Hausdorff dimension of the degree measure is  
 always smaller than $1$ (since the equality case happens only when $\deg(\sigma)=0$). 

Theorem \ref{theo:formula} shows that the mass of the degree defines a psh function on the moduli space of $\mathbb P^1$-structures on $X$. It would be interesting to know if the degree is also psh considered as a measure. 

\end{rmk}

\subsection{Projective structures with vanishing degree}\label{ss:vanishing}
 Recall that $\sigma$ is a {\em covering projective structure} on $X$ if its developing map is a covering of
  some proper open subset $\om\subset\pu$. In other words, if $\sigma$ is the quotient of the orbit of a component of the discontinuity set of a Kleinian group. Such projective structures were studied e.g. by 
  Kra \cite{kra1, kra2} who showed  that a parabolic projective structure is of covering type if and only if its developing map is not surjective.

  \medskip 
  
Projective structures of covering type may also be  characterized in terms of their  degree. 

\begin{prop} \label{p:vanishing}
Let $X$ be a Riemann surface of finite type and 
 $\sigma$ be a parabolic projective structure on $X$. 
 Then $\mathrm{deg}(\sigma) = 0$ if and only if $\sigma$ is of covering type.
\end{prop}

\begin{proof} The proof of  Definition-Proposition~\ref{defprop:degree} shows that the degree vanishes if and only if the support of 
the foliated harmonic current $T$ is disjoint from $s$. Equivalently, the image of the developing map is disjoint from the support of the harmonic measures. Hence the developing map is not surjective and the result follows from the remarks preceding the proposition.
\end{proof}

\subsection{Cohomological expression of the degree}\label{ss:cohomological degree}

Observe that the harmonic current $\overline{T}$  on $\overline M_\sigma$ 
naturally defines an element of the dual of the Bott-Chern cohomology group 
$H^{1,1}_{\partial \overline{\partial}} (\overline{M}_\sigma, \mathbb C)$. By the $\partial \overline{\partial}$-lemma, 
the natural map $ H_{\fr\overline \fr}^{1,1} (\overline{M}_\sigma,\mathbb C) 
\rightarrow H_{\overline{\partial}} ^{1,1} (\overline{M}_\sigma, \mathbb C)$ with values in the Dolbeaut cohomology group is an isomorphism. 
Thus by duality, 
the current $\overline{T}$ defines a cohomology class $[\overline{T}]$ in $H^{1,1} (\overline{M}_\sigma,\mathbb C)$. 
In more concrete terms, if $\alpha_1$ 
and $\alpha_2$ are smooth (1,1) forms defining  the same class $[\alpha]$ in $H^{1,1} 
(\overline{M}_\sigma,\mathbb C)$, then $\alpha_1 - \alpha_2 = dd^c u$ 
for some smooth function $u$, and we get that  $\bra{\overline T, \alpha_1} = \bra{\overline
T, \alpha_2}$. So it makes sense to speak about the pairing
between $\overline T$ and $\alpha$ which we simply denote it  by $\overline T\cdot \alpha$. 
Observe also that any
 curve $C\subset M_\sigma$ admits a class in $H^{1,1}(M_\sigma)$, which is the cohomology class dual to the cycle $C$ 
 (or equivalently, that of the integration current $[\overline s]$).  
 
\medskip

Recall from \S\ref{ss:existence of degree} 
that $\deg(\sigma)$ is the mass of $\overline T\geom{\overline s}$. 
The next result --presumably  part of the folklore--
 asserts that this geometric intersection number    can be computed in cohomology. 
 
\begin{prop} \label{p:cohomological expression of the degree}
Let $\sigma$ be a parabolic projective structure on a Riemann surface of finite type, 
and $\deg(\sigma)$ be be its degree, as defined in Definition \ref{defprop:degree}. Then   
$$ \deg(\sigma)  = \overline{T} \cdot \overline{s}.$$
\end{prop}

\begin{proof}
The difficulty is that we cannot simply regularize the integration current $[s]$ within smooth positive forms because, 
as we will see later, $\overline s^2 <0$. Pick a smooth closed $(1,1)$ form cohomologous to 
$[\overline s]$, and write $[\overline s]=\alpha + dd^cu$, where $u$ is a quasi-psh function, smooth outside $\overline s$,  
 with logarithmic singularities along $s$. 
Then by definition, $\overline T \cdot \overline s  = \bra{T, \alpha}$. 
Recall that $s$ stays far from the singularities of the foliation  $\overline {\mathcal F}$ and is everywhere transverse to it. 
Consider a tubular neighborhood ${N}_\e$ of $\overline s$, such that if $p\in \overline s$ and 
$L_p$ is the leaf through $p$, then $L_p\cap N_\e$ is a small disk about $p$, contained in a flow box. 
We modify $u$ by replacing it inside $N_\e$ by any 
smooth function $u_\e$  such that  $u = u_\e$ near $\fr N_\e$.  We denote by $u_\e$ the resulting function on $\overline M_\sigma$. By construction, $[\overline s]_\e:= \alpha + dd^cu_\e$ is a smooth form cohomologous to $[\overline s]$, so 
$  \overline T\cdot  [\overline s] =\bra{ \overline T,   [\overline s]_\e} $. 

Now consider a flow box $\bb$ endowed with  local coordinates  $(z,w)\in \dd^2$  where $\overline{\mathcal{F}}$ becomes the horizontal foliation and $\overline s$ is a vertical graph. 
Then  in this flow box,  $[\overline s] = dd^c v $ for some psh function $v$ and 
$[\overline s]_\e = dd^cv_\e$ with $v=v_\e$ in a neighborhood of 
 $\fr \dd\times \dd$. With  notation as in \eqref{eq:foliated}, we see that the 
local contribution of $\bra{ \overline T,   [\overline s]_\e} $ is equal to 
$$\bra{ \overline T,   [\overline s]_\e}\rest{\bb}  =  \int \lrpar{ \int_{\dd\times \set{w} }  \varphi dd^c v_\e } dm(w)  =   
\int \lrpar{\int_{\dd\times \set{w} }  \varphi dd^c v } dm(w)   = T\geom \overline s \rest{\bb}, $$
where the middle equality follows from the Green formula and the harmonicity of $\varphi$. The result follows.
\end{proof}






\section{The Lyapunov exponent} \label{sec:lyapunov}

In this section we relate the exponent $\chi$ defined in Definition \ref{def:lyap} to a foliated Lyapunov exponent introduced by the first author in \cite[Appendice]{deroin levi plate}. This leads in     \S\ref{ss:cohomological exponent} 
to a cohomological formula for $\chi$ analogous to that obtained for  the degree. 

\subsection{The foliated Lyapunov exponent}
 Using    \cite[Proposition 2.2]{Bers1}, we start by introducing  a Lipschitz family of spherical 
 metrics on $M_\sigma$, simpy denoted by 
 $\norm{\cdot}$. By this, we mean a
  smooth family of  conformal metrics of curvature $+1$ on the fibers, with the property that there exists $C>0$ such that 
  for every  smooth path $\omega: [0, 1]\cv X$, 
 $\log \norm {D h_\rho(\omega)}_\infty \leq C \mathrm{length}(\omega)$, where $h_\rho(\omega)$ is the holonomy of $\omega$ and 
 $\norm {D h_\rho(\omega)}_\infty$ is the supremum of 
 the norm of the fiber derivative   relative to  the    spherical metrics 
 on $\pi^{-1}(\omega(0))$ and $\pi^{-1}(\omega(1))$. We will recall some details of the construction below in \ref{ss:cohomological exponent}. 
 More generally, the notation $\mathrm{length}(\omega) $
 will stand for  the {\em homotopic length} of $\omega$, that is, the
 minimal length of a smooth path homotopic to $\omega$ with fixed endpoints. In particular this notion makes perfect
 sense for a Brownian sample path. 
 
   Since $\mathcal{F}$ is transverse to the fibers,   this
  induces a smooth metric on the normal bundle $N_{\mathcal F}$. 
  Later on we will study  the extension properties of $\norm{\cdot}$ to a singular metric on the fibers of $\overline M_\sigma$.
  
 Notice that in our situation, the data of a Brownian sample path along a leaf is equivalent to that of its projection on $X$, together with its starting point in the initial fiber. So if the starting point $x$ is given, the projection $\pi$ gives an identification between 
  $W^{\mathcal{F}}_x$ and $W_{\pi(x)}^X$.  
In this way we can speak of the holonomy, or homotopic length  of a leafwise 
 Brownian path, by simply projecting it  to $X$.  
 
 We now consider the family 
of functions  
$$\Omega^\mathcal{F}\ni\omega\longmapsto K_t ({\omega}) = \log \norm{D _{\omega(0)} h (\omega\rest{[0,t] }) }.$$ This 
is a cocycle, in the sense that  $K_{t+s} (\omega) = K_s(\omega) + K_t (\sigma_s \omega)$ for every $t,s\geq 0$. 
As explained above, the estimate    
$$  { K_t({\omega}) } \leq C \cdot \text{length} (\omega\rest{[0,t]}) ,$$
holds, for some $C$ is independent of $\omega$. 
The superexponential decay of the heat kernel on the hyperbolic plane  \cite[\S 5.7]{davies} then implies 
 that $K_t$ is $W^{\mathcal F}_\mu$-integrable for every $t\geq 0$.  The ergodic theorem shows that for 
$W^{\mathcal F}_\mu$-almost every path $ {\omega}$ the limit $\la = \lim_{t\cv\infty}\frac{K_t ({\omega}) } {t}$ exists and does not depend on $\omega$. By definition $\la$ is the {\em foliated Lyapunov exponent}. 

\medskip

We can now compare $\la$ and $\chi$. 

\begin{prop}\label{p:comparison exponents} 
Let $\sigma$ be a parabolic projective structure on a Riemann surface of finite type. 
Let $\chi(\sigma) = \chi_{\rm Brown}(\mathsf{hol}_\sigma)$ be the Lyapunov exponent of $\sigma$, as defined in \S \ref{ss:lyapunov}. Then if $\la$ is as above we have
$\lambda  = -2 \chi (\sigma) $.  \end{prop}

 The proof relies on the following result:

\begin{lem}\label{l:technical step} Assume that $\rho$ is non elementary. Then for every $x\in X$ and $W^X_x$-a.e. $\omega: [0,\infty) \rightarrow X$ starting at $x$, there exists $r( \omega) \in \mathbb P^1 _x$ such that the pointwise convergence
\begin{equation}\label{eq:norm} \lim_{t\rightarrow \infty} \frac{1}{t} \log \norm{ D_y h (\omega\rest{[0,t]}) } 
\rightarrow -2\chi_{\rm Brown}(\rho)  \end{equation} holds
uniformly on compact subsets of $\mathbb P^1 _x \setminus \{r(\omega) \}$. Moreover, the distribution  of the exceptional 
point $r(\omega)$ is the harmonic measure $\nu_x$ on $\mathbb P^1_x$. \end{lem}

\begin{proof} 
In order to apply the Oseledets theorem, 
consider a measurable trivialization $M_\sigma \simeq X \times \mathbb P^1$ 
 and set $\chi = \chi_{\rm Brown}(\rho)$. 
For every continuous $\omega$, and every $t> 0$, the map $h_t = h_{\omega\rest{[0,t]}} : \mathbb P^1_{\omega(0)} \rightarrow \mathbb 
P^1 _{\omega(t)}$ can be lifted to a matrix $\widetilde{h_t}$ in $\text{SL}(2,\mathbb C)$ which is well defined up to sign. The family     $
\widetilde{h} = \{ \widetilde{h_t} \}_{t\geq 0}$ on $\Omega$ is a cocyle modulo signs,  namely it satisfies $\widetilde{h_{t+s}} (\omega) = \pm \widetilde{h_t} (\sigma _s(\omega)) \widetilde{h_s} (\omega)$ for 
every $\omega \in \Omega$ and every $s,t\geq 0$.  Moreover, from  \cite[Proposition 2.5]{Bers1}, we have that for every $x\in X$ and 
$W^X_x$ a.e. $\omega$,
\[  \lim _{t\rightarrow \infty} \frac{1}{t} \log \norm {\widetilde{h_t}}  = \chi ,\] 
where $\norm{\cdot}$ is the matrix norm to the usual hermitian norm $\norm{\cdot}_2$ on $\cd$. 
Since  $\rho$ is non elementary, by \cite[Thm 2.7]{Bers1}, $\chi>0$. Since in addition 
   $h$ takes values in $\text{SL}(2,\mathbb C)$, the Lyapunov exponents of $h$ over $(\Omega_X , \sigma_X , W^X)$ 
   are $\chi $ and $-\chi$. The  Oseledets theorem  tells us  that for $W$-a.e. $\omega: [0,\infty) \rightarrow X$, there exists a complex line $E= E(\omega) \subset \mathbb C^2$ such that for every $Y\in \mathbb C^2$, $Y\neq 0$, 
      $ \frac{1}{t} \log \norm{\widetilde{h_t} (Y)}_2 $ converge to $-\chi$ as $t\cv\infty$ when $Y\in E$,  
      while this quantity 
       converges uniformly to $\chi$ on compact subsets of $\mathbb C^2 \setminus E$.  
   Finally, we observe that for the usual spherical derivative, we have
   that 
\begin{equation} \label{eq:spherical derivative} \norm {D h_t (y) }_s    = \frac{\norm{Y}_2^2}{\norm{\widetilde{h_t}(Y)}_2^2}, \text{ where $Y$ is a lift of $y$},\end{equation} 
hence   \eqref{eq:norm} holds, 
with $r(\omega) = \mathbb P E(\omega) \in 
\mathbb P^1_{\omega(0)}$. 

It remains to show  that the distribution of $r$ when $\omega$ is  conditioned  to start at $x$ is the harmonic measure $\nu_x$. 
For this, we consider  the  mapping $\omega^\mathcal{F}\ni\omega \mapsto r(\omega)$, which is defined $W_\mu^\mathcal{F}$-a.e.
 The push-forward of $W_\mu^\mathcal{F}$  is a shift invariant measure on $M_\sigma$, so we conclude 
 by the  unique ergodicity of $\mathcal{F}$ (Proposition \ref{p:ergodic}).
\end{proof}

\begin{proof}[Proof of  Proposition \ref{p:comparison exponents}] 
Let ${x}\in M_\sigma$. 
As observed before, we can identify $(\Omega^X _ {\pi(x)}, W^X_{\pi(x)})$ in $X$ and 
$(\Omega^{\mathcal F}_{ {x}}, W^{\mathcal F}_{ {x}})$ in $M_\sigma$ by 
 lifting. 
Since the harmonic measure $\nu_{\pi(x)}$ on $\mathbb P^1_{\pi(x)}$ has no atoms, we infer that 
for $W_{{x}}^\mathcal{F}$ a.e. $\omega$, the 
point $r(\omega)$ defined in  Lemma \ref{l:technical step} is distinct from ${x}$. 
Hence $ \lim_{t\rightarrow \infty} \frac{1}{t} \log \norm{ D_{x} h (\omega\rest{[0,t] }) }  \rightarrow -2\chi$
for $W_{{x}}$-a.e. $\omega$,
 and  the conclusion follows. 
\end{proof}

\subsection{Cohomological expression of $\chi$}\label{ss:cohomological exponent}
Let  $P$ be the set of punctures of $X$. To avoid confusion with the Lyapunov exponent, we denote by $\mathrm{eu}(X)$ the Euler 
characteristic of $X$, $\mathrm{eu}(X)  = 2-2g -\# P$. 
Recall   
  the Gauss-Bonnet formula  $\vol(X) = 2\pi \abs{\mathrm{eu}(X)}$. 

 In this section, we begin the proof of the following result, which will be complete only after proving Theorem \ref{theo:formula}

\begin{prop} \label{p:computation of Lyapunov exponent}
Let $\sigma$ be a parabolic projective structure on a Riemann surface of finite type with puncture set $P$. Then 
$\chi (\sigma)  = \displaystyle \frac{1}{2\abs{\mathrm{eu}(X)} } ( N_{\overline{\mathcal F}} \cdot \overline{T} +\#P)$.
\end{prop}

When $X$ is compact ($P = \emptyset$)
 this result follows from the  cohomological formula derived in \cite[Appendice A]{deroin levi plate} for the foliated Lyapunov exponent, and from Proposition \ref{p:comparison exponents}. The  proof in the non compact case follows the same strategy 
 but serious technical difficulties arise from the parabolic cusps. 

Recall that if $X$ is a complex surface, $E\rightarrow X$ is a holomorphic line bundle, and $\norm{\cdot}$ is a hermitian metric on $E$, its curvature form is defined by $\Theta(\norm{\cdot}) = \frac{1}{2 i \pi }\partial \overline{\partial} \log \norm{s }^2 $, where $s$ is any non vanishing local holomorphic section of $E$. 
In our situation we choose  
a Lipschitz family of spherical metrics on the fibers of $M_\sigma$, which, since $\mathcal{F}$ is transverse to the fibers, 
 induces   a  hermitian metric on the normal bundle $N_{\mathcal F}$. 
Recall that the value of the Lyapunov exponent does not depend on this choice.   
We denote by $\Theta$ the curvature form of this metric.

The first result is obtained exactly as in the compact case \cite[Appendice A]{deroin levi plate} (see also \cite[\S 8]{candel}).

\begin{lem}\label{l:computation lyapunov exponent}
$\displaystyle \chi (\sigma) = \frac{\pi}{\vol (X)} \int \Theta \wedge T$.
\end{lem}

\begin{proof} 
 We keep notation as in \S\ref{ss:lyapunov}.  From the fact that $K$ is a cocycle, 
 we deduce that the function $t\mapsto \int K_t(\omega) W^\mathcal{F}_\mu(d\omega)$ is linear. 
 So its    slope is equal to its derivative at 0, and we get that 
\begin{equation} \label{eq: formula lyapunov exponent} \lambda= \frac{d}{dt} \Bigl\vert_{t=0} 
\int K_t (  {\omega} )  W_\mu^\mathcal{F}(d\omega) 
= \int _{M_\sigma} \frac{d}{dt}\Bigl\vert_{t=0} \mathbb E^{ {x}} \big( K_t ( {\omega} ) \big) d\mu ({x}) . \end{equation}
Let $x_0$ be a point of $X$. We use  local coordinates $x=(\xi, \eta)$, 
to parametrize points in $M_\sigma$ via $s(\xi, \eta)$, that is,  ${x}$ belongs to the fiber of $\xi$, $\eta$ 
belongs to a neighborhood of $\eta_0$ in $\mathbb P^1_\xi$.
and $s(\xi, \eta) = h _{\xi, \eta} ({x})$ is the flat section passing through the point ${x}$, 
defined over a neighborhood of $\xi_0$ in $X$.  Using the heat equation, the formula~\eqref{eq: formula lyapunov exponent} can be written in these coordinates 
\[ \lambda = \int _{X_\rho} \Delta_{\xi} \log \norm{\frac{\fr}{\fr \eta} h_{\xi,\eta}  ({x})}\ \mu ( d{x}) \] 
Observe that the curvature form $\Theta$ of the Lipschitz metric on $N_{\mathcal F}$, restricted to the tangent bundle of $\mathcal F$, is given by the expression 
$$\Theta \rest{T\mathcal F}  =  \frac{1}{2 i \pi }\partial \overline{\partial}_\mathcal{F} \log \norm{\frac{\fr}{\fr \eta}   h_{\xi,\eta}  ({x} )}^2 .$$ 
Because we have $\Delta_{\rm Poin} f \cdot \text{vol} _{\rm Poin} = 2 i \partial \overline{\partial} f$ for every function $f$ defined on the hyperbolic plane, we infer that 
\[ \Delta_{\xi} \log \norm{ \frac{\fr}{\fr \eta}   h_{\xi,\eta}  ({x} )} \text{vol} _{\rm Poin}   = - 2\pi\ \Theta\rest{T\mathcal F} . \]
Using the fact that $T \wedge \text{vol}_{\rm Poin} = \text{vol(X)} \mu$, we finally obtain 
\[  \lambda = -\frac{2\pi}{\text{vol} (X)}  \int \Theta \wedge T, \]
which,  together with Proposition~\ref{p:comparison exponents}  finishes the proof of Lemma \ref{l:computation lyapunov exponent}. \end{proof}

When  $X$ is compact, it immediately follows from Lemma~\ref{l:computation lyapunov exponent} 
 that  
\begin{equation} \label{eq:cohomological formula compact case} \chi =\frac{\pi} {\text{vol}(X)} T\cdot N_{\mathcal F} = 
\frac{1}{2\abs{\mathrm{eu}(X)}}   N_{ {\mathcal F}} \cdot  {T}, \end{equation}
and the proof of Proposition \ref{p:computation of Lyapunov exponent} is complete. 

In the general case, however, this calculation is no longer valid, and  in the remaining part of the argument we need to 
understand  the contribution of the punctures to this formula. For the moment, we content ourselves with the following weakening of
Proposition \ref{p:computation of Lyapunov exponent}. 

\begin{prop}\label{p:weaker}
Under the assumptions of Proposition \ref{p:computation of Lyapunov exponent}, there exists  a universal constant $I$ such that 
$\chi (\sigma)  = \displaystyle \frac{1}{2\abs{\mathrm{eu}(X)} } ( N_{\overline{\mathcal F}} \cdot \overline{T} + I \cdot \#P)$.
\end{prop}

The proof occupies the remainder of this section. It will be carried out in several steps, mostly  dealing with the local study of the model foliation $\mathcal{F}_m$ introduced in \S \ref{ss:model}.

\medskip

\noindent{\em Step 1. A smooth metric.}

Let $p\in P$ be a puncture of $X$, and let us work in a neighborhood $\pi^{-1}(U(p))$
of $\pi^{-1}(p)$ in $\overline M_\sigma$,  in the coordinates
$(u,v)$ introduced in \S \ref{ss:model}.  We claim that the metric 
\begin{equation} \label{eq:smooth metric} 
\norm{\cdot}_s = |u| \frac{|dv|}{1+|v|^2}  
\end{equation} 
defines   a smooth metric on $N_{\overline{\mathcal F}}$. To see this, observe 
that a non-vanishing holomorphic section of the normal bundle of 
$\overline{\mathcal F} = \mathcal{F}_m$ on $\mathbb D \times \mathbb C$ in the $(u,v)$-coordinates 
is defined by $n=\frac{1}{u} \frac{\partial}{\partial v}$. Indeed, $\omega = 
du + 2i \pi  u dv$ is a form defining $\overline{\mathcal F}$, and $\omega (n) = 2i\pi$. We 
see that  $\norm{n}_s = \frac{1}{1+|v|^2}$, so   $\norm{\cdot}_s$ extends smoothly along the line 
$\set{0}\times \mathbb C $. To analyse what happens close to the point $(0,\infty)$, we introduce the new coordinates 
$(u,V) = (u,\frac{1}{v})$. In these coordinates, the foliation is defined 
 by the equation $2\pi u dV + i V^2 du = 0$. A non-vanishing section of the normal bundle is then given by $n = \frac{1}{u} \frac{\partial} {\partial V}$, and a straighforward computation yields $\norm{\cdot}_s = |u| \frac{|dV|}{1+|V|^2}$. 
 Hence the situation is symmetric and we conclude that $\norm{\cdot}_s$ 
 defines a smooth metric on $N_{\overline{\mathcal F}}$, as claimed. 

\medskip

\noindent{\em Step 2. The Lipschitz metric. } 

 Here we give an explicit expression for 
 a Lipschitz family of spherical metrics on $ M_\sigma$ close to $p$. Recall that a model for the bundle $\pi^{-1}(\mathcal U(p))\subset  M_\sigma$ is the quotient of $\mathbb H \times \mathbb P^1$ by the identification 
  $(\tau, z) \sim (\tau + 1 , z + 1)$. A Lipschitz family of spherical metrics on this model is defined by
\begin{equation} \label{eq:Lipschitz spherical metric} \norm{\cdot}_{\tau} = \frac{\Im \tau \ |dz|}{|z-\Re \tau|^2 + \Im^2 \tau }  \end{equation}
It is constructed by starting with the spherical metric $\norm{\cdot}_i = \frac{|dz|}{1+|z|^2}$, which is already invariant by the stabilizer 
$\text{PSO}(2,\mathbb R)$ of the point $i$, and then by extending it by the formula $M^* \norm{\cdot}_{M\tau} = \norm{\cdot}_\tau$ for any $
\tau\in \mathbb H$ and $M\in \mathrm{PSL}(2,\mathbb R)$. The proof of \cite[Prop. 2.2]{Bers1} shows that $\norm{\cdot}_\tau$ is indeed 
Lipschitz. 

The family $\{  \norm{\cdot}_\tau \}_{\tau\in \mathbb H}$ then induces a family of spherical metrics $\{  \norm{\cdot}_u \}_{u\in \mathbb D^*}$ on the quotient bundle $\simeq \mathbb D ^* \times \mathbb P^1$  which is given by the formula 
\begin{equation} \label{eq:Lipschitz spherical metric 2} \norm{\cdot}_u = \frac{\frac{1}{2\pi}\log \big( \frac{1}{|u|}\big) |dv|}
{|v+ \frac{i}{2\pi}\log \big( \frac{1}{|u|}\big) |^2 + \frac{1}{4\pi^2}\log^2 \big( \frac{1}{|u|}\big)} 
=  \frac{2\pi} {\log \big( \frac{1}{|u|}\big)} \cdot \frac{|dV|}{|\frac{2\pi }{\log (\frac{1}{|u|})} + iV| ^2 + |V|^2}.\end{equation}

\medskip

\noindent{\em Step 3. The induced singular metric on $N_{\overline{\mathcal F}}$.} 

The  family of spherical metrics constructed above on $ M_\sigma$ induces a metric 
 $\norm{\cdot}$ on the normal bundle of the foliation $\overline{\mathcal F}$ which possesses 
singularities along the fibers over the cusps of $X$, that we compute here. In the $(u,V)$-coordinates, we have that  
\begin{equation}\label{eq:singularities}  \frac{\norm{\cdot}}{\norm{\cdot}_s}  = \frac{2\pi }{|u| \log (\frac{1}{|u|})} \cdot \Phi(u,V)
\text{ where } \Phi(u,V) = \frac{1+ |V|^2}{ | \frac{2\pi }{\log(\frac{1}{|u|})}  + i V|^2 + |V|^2}.
 \end{equation} 
  The reader can check that $\Phi$ has a pole only at the point $(u,V)=(0,0)$,  extends continuously and extends continuously  elsewhere. 

\medskip

\noindent{\em Step 4. Defining a foliation index. } 
For {\em any} harmonic current $T$ on $\mathbb D \times \mathbb P^1$ directed by $\mathcal F _m$, we define
\begin{equation}\label{eq:index}  I(T):= \int_{\mathbb D^* \times \mathbb P^1} \frac{1}{i\pi} \partial \overline{\partial} \Psi \wedge T \end{equation}
where $\Psi : \mathbb D^* \times \mathbb P^1 $ is a smooth function supported in 
a domain $  D_r ^* \times \mathbb P^1$ for some $0<r<1$, and such that $\Psi = \log \frac{\norm{\cdot}}{\norm{\cdot}_s}$ in a neighborhood of $0\times \mathbb P^1$. 
Observe that this number does not depend on the chosen function $\Psi$, since the current $T$ is harmonic. 

\begin{lem} \label{l:convergent}
The integral~\eqref{eq:index} is convergent.  
\end{lem}

\begin{proof} 
It suffices to proves the lemma for  $\Psi = \log \frac{\norm{\cdot}}{\norm{\cdot}_s}$. In this case
 the integral $I(T)$ is nothing but the $T$-integral of the differences between the curvature of $\norm{\cdot}$ and that of 
 $\norm{\cdot}_s$. Because $\norm{\cdot}_s$ is smooth and hence $T$-integrable, it is enough
  to prove that the curvature of $\norm{\cdot}$ is $T$-integrable. We claim that the restriction of the curvature of $\norm{\cdot}$ along 
  the leaves is bounded in modulus by the leafwise Poincar\'e metric. 
  This is sufficient for our purposes since the Poincar\'e metric is $T$-integrable 
  (due to the fact that $T$ projects on the integration current on $\mathbb D$) and that the $T$-integral of a $(1,1)$-form depends 
  only on its restriction to $\mathcal F$. 
  
  To prove this claim, we work in the $(\tau,z)$-uniformizing coordinates, 
  and use   formula~\eqref{eq:Lipschitz spherical metric} to get that the curvature of 
  $\norm{\cdot}$ along the leaf $\mathbb H^2 \times z$ is 
\[  \frac{1}{i} \partial \overline{\partial}_\tau \log \big( \frac{\Im \tau }{|z-\Re \tau|^2 + \Im^2 \tau } \big) .\]
Then,  writing $\tau = x+iy$, we compute
\[ \frac{1}{i} \partial \overline{\partial}_\tau \log \big( \frac{ y }{|z- x|^2 + y^2  } \big) = \big( \frac{-1}{y^2} + \frac{2 \Im ^2 z }{\big(  (x- \Re z) ^2 + y^2 + \Im^2 z\big)^2 } \big) dx\wedge dy \]
and the result follows since 
\[ 0\leq \frac{2 \Im ^2 z }{\big(  (x- \Re z) ^2 + y^2 + \Im^2 z\big)^2 } \leq \frac{2}{ y^2}. \] 
\end{proof}

The index is defined so as to have the following formula, which corrects formula \eqref{eq:cohomological formula compact case}. For every puncture $p$ of $X$, we define $I(T,p)$ to be the index of the
canonical foliated harmonic current defined in subsection \ref{ss:harmonic currents} at the puncture $p$. 

\begin{lem}   \label{l:lyapunov exponent 1}
$\displaystyle\chi(\sigma)  = \frac{\pi} {\mathrm{vol}(X)} \left( N_{\overline{\mathcal F}} \cdot T + \sum _p I(T, p) \right)$.
\end{lem}

\begin{proof}
Let $\Psi$ be the function on $M_\sigma$, defined in a neighborhood of the exceptional fibers, 
as just constructed. Introduce a smooth family of 
metrics  on the fibers of  $\overline M_\sigma$, which coincides with $\norm{\cdot}_s$ near the punctures. Such a family is not Lipschitz, so we multiply it by a function of the form $e^\Psi$, to make it coincide with the local model discussed above, 
 $\norm{\cdot} = \norm{\cdot}_s  \cdot e^\Psi$. Then we infer that  
\[  \int \Theta \wedge T = \int \Theta_{\norm{\cdot}_s} \wedge T + \int \frac{1}{i\pi} \partial \overline{\partial} \Psi \wedge T= N_{\overline{\mathcal F}} \cdot T + \sum _p I(T, p) \]
and result follows from  Lemma \ref{l:computation lyapunov exponent}.
\end{proof}

\medskip

\noindent{\em Step 5. An invariance property for the index}

\begin{prop} \label{p:invariance}
The index $I(T)$ takes the same value on all foliated  harmonic currents $T$ on $\mathcal F_m$ that give mass $1$ 
to the  fibers $\set{u}\times \mathbb P^1$. 
\end{prop}

\begin{proof} 
Let us introduce two families of symmetries for the foliation $\mathcal F_m$. They are 
 induced by the translations $(\tau,z) \mapsto (\tau +x, z)$ and $(\tau,z) \mapsto (\tau, z+c)$ for $x\in \mathbb R$ and $c\in \mathbb C$ at the level of the universal cover~:
\begin{equation} \label{eq:symmetries} H_x(u,V) = (e^{2i\pi x} u, \frac{V}{1-xV})\ \ \ \text{and}\ \ \ V_c (u, V) = (u, \frac{V}{1+cV}). \end{equation}
The following result  is the key of the argument:

\begin{lem}\label{l:invariance} If $T$ is as in Proposition \ref{p:invariance}, then for all $x\in \mathbb R$ and $c\in \mathbb C$,  
$I(( H_x)_*T)=I((V_c)_* T)= I(T)$.
\end{lem}

\begin{proof}
We treat the case of $H_x$, the proof being similar (and in fact easier) for $V_c$. We have that
\[ I(( H_x)_*T)- I(T)  = \int_{\mathbb D^* \times \mathbb P^1} \frac{1}{2i\pi} 
\partial \overline{\partial}  (\Psi \circ H_x - \Psi) \wedge T .\]
Let us split this function as a sum  
\begin{equation}\label{eq:splitting} \Psi \circ H_x - \Psi = \Gamma + \Gamma _s, \end{equation}
where $\Gamma$ and $\Gamma_s$ are smooth functions on $\mathbb D^* \times \mathbb P^1$   supported in 
 $\mathbb D^* _r\times \mathbb P^1$ for some $0<r<1$ and such that 
 in a neighborhood of the divisor $\set{u=0}$,   $\Gamma = \log \frac{(H_x)_*\norm{\cdot}}{\norm{\cdot}}$ and $\Gamma_s =\log \frac{(H_s)_*\norm{\cdot}_x}{\norm{\cdot}_s}$. 
Observe that 
 \begin{equation}\label{eq:gamma}
 \int  \frac{1}{2i\pi} \partial \overline{\partial}  \Gamma_s \wedge T =0
 \end{equation} since $\norm{\cdot}_s$ is a smooth metric,
  thus $\Gamma_s$ is a smooth function. 
Now   $\Gamma$ is smooth if $u\neq0$, tends to $0$ uniformly when $u$ tends to $0$, and
    the derivative of $\Gamma$ along the leaves is bounded by the Poincar\'e metric (since $\norm{\cdot}$ is a Lipschitz metric). Since the leafwise Poincar\'e metric is given in $u$-coordinates by $\frac{|du|}{|u|\log |u|}$, we get that 
\begin{equation} \label{eq:bound}  | d_{\mathcal F} \Gamma |\leq \frac{|du|}{|u|\log |u|} .\end{equation} 
From  \eqref{eq:splitting} and \eqref{eq:gamma} we are left to prove that 
\begin{equation} \label{eq:computation} \int \frac{1}{2i\pi}  \partial \overline{\partial}  \Gamma \wedge T =0 \end{equation}
(the fact that this integral makes sense follows from Lemma \ref{l:convergent}).
To do this, we introduce a family of smooth functions $\theta_r: \mathbb D^* \rightarrow [0,1]$ such that $\theta_r(u)=1$ if $|u|\geq r$, $\theta_r (u) = 0$ if $|u|\leq r/2$, $\norm{d\theta_r }_{\infty} = O (\frac{1}{r})$, and $\norm{\partial \overline {\partial} \theta_r}_{\infty} = O(\frac{1}{r^2})$. Since $\fr\overline \fr \Gamma\wedge T$ is of order 0, 
 to get~\eqref{eq:computation}, it is enough to prove that 
\[  \lim_{r\rightarrow 0} \int \theta_r  \partial \overline{\partial}  \Gamma \wedge T = 0 . \]
To compute this integral, we observe that since $T$ is harmonic
$\int \partial \overline{\partial} (\theta_r \Gamma) \wedge T = 0$ , hence we get  that 
$$ \int \theta_r \partial \overline{\partial}  \Gamma \wedge T = - \int \Gamma \partial \overline{\partial}  \theta_r \wedge T  - 
2\Re \int \partial \theta_r \wedge \overline{\partial} \Gamma\wedge T  =: - A_r  - B_r.
$$
%
To conclude the proof, we will show   that both integrals $A_r$ and $B_r$ tend to $0$ with $r$. 
 To estimate the former, we write  
\[  \abs{A_r}  \leq \delta(\Gamma,r) O (\frac{1}{r^2}) \int _{ \frac{r}{2} \leq |u|\leq r} idu\wedge d\overline{u} \wedge T \leq O (\delta(\Gamma, r) )\]
where $\delta(\Gamma,r) = \sup _{\frac{r}{2}\leq |u| \leq r, v\in \mathbb P^1} \Gamma(u,v) $, and the last inequality holds because $T$
 projects on the current of integration on $\mathbb D$. As observed above,  $\delta(\Gamma,r)= o(1)$ whence 
$ \lim_{r\rightarrow 0} A_r = 0 $. The same argument works for the second integral: indeed by using~\eqref{eq:bound} and the bound on 
$\norm{d\theta_r }_{\infty}$, 
we get that 
\[  \abs{B_r} \leq O(\frac{1}{r^2 \log (\frac{1}{r})})  \int _{ \frac{r}{2} \leq |u|\leq r} idu\wedge d\overline{u} \wedge T \leq O(\frac{1}{\log (\frac{1}{r})}) ,\]
which completes the proof.
\end{proof}

Let us resume the proof of Proposition \ref{p:invariance}. 
Recall from \S \ref{ss:harmonic currents} that a foliated harmonic current $T$ for $\mathcal{F}_m$ lifts as a harmonic current 
$\widetilde{T}$ on $\hh\times \pu$, which  by the Poisson formula is 
induced by a family of probability measures 
$\set{\nu_a}_{a\in \rr}$ on $\pu$, depending measurably  on $a$, and 
satisfying the relation $\nu_{a+1} = (z+1)_*\nu_a$.  From this equivariance, the data of such a family of measures is in turn equivalent 
to that of a probability measure on $[0,1)\times \pu$. Such a measure is a convex combination of Dirac masses on $[0,1)\times \pu$. 
This shows that any family $\set{\nu_a}$ as above is a convex combination of families of the form $\nu_{{(a_0, z_0)}}$, $a\in [0,1)$, $z_0\in \pu$, 
where $(\nu_{(a_0, z_0)})_a = 0$ if $a\neq a_0$ mod. $\zz$ and $(\nu_{(a_0, z_0)})_{a_0+k} = \delta_{z_0+k}$. 
(Notice that the point $\infty\in \pu$, corresponding to the separatrix of the singularity of $\mathcal{F}_m$,
 plays a special role here. Nevertheless we do not need to take it in to account since our measures and currents are diffuse.)
Going back to currents, 
 $\nu_{(a_0, z_0)}$ corresponds 
to a certain harmonic current $T_{(a_0, z_0)}$ and all foliated harmonic currents for $\mathcal{F}_m$ are obtained from these by 
taking convex combinations. 

Now it is clear that $(H_x)_*(T_{(a_0, z_0)}) = T_{({a_0+x, z_0})}$ and 
$(V_c)_*(T_{(a_0, z_0)}) = T_{({a_0, z_0+c})}$, hence we infer from  Lemma \ref{l:invariance} that  the index $I$ 
takes the same value on all the extremal points $T_{(a_0, z_0)}$, and we are done. 
\end{proof}

\section{Proof of Theorem \ref{theo:formula} (and of Proposition \ref{p:computation of Lyapunov exponent})}\label{sec:proof}

The proof is based on some basic cohomological computations in $H^2(\overline{M }_\rho, \cc)$.
Recall  from \S\ref{ss:generalities} that $H^2(\overline{M }_\rho, \cc) = \cc\overline [\overline{s}] \oplus \cc   [f]$. 
We will need the following fact: if $\mathcal G$ is a singular holomorphic foliation on a complex surface, and $C $ is a non singular compact holomorphic curve 
that is everywhere transverse to $\mathcal{G}$, then 
$$ N_{\mathcal G} \cdot C = \mathrm{eu} (C).$$
Indeed, under these assumptions, $N_\mathcal{G}\rest C \simeq T_C = -K_C$, 
and by the genus formula, $K_C\cdot C = -\mathrm{Eu(C)}$. 
Hence in our situation, working in $\overline M_\sigma$ we get that 
\begin{equation}\label{eq:npoint}
N_{\overline{\mathcal{F} } }\cdot \overline s = \mathrm{eu} (\overline s) = \mathrm{eu} (\overline X) \text{ and } 
N_{\overline{\mathcal{F}}}\cdot f = \mathrm{eu}(\pu) = 2.
\end{equation}
The intersection form in $H^2(\overline{M }_\rho, \cc)$ is characterized by the identities 
$$\overline s \cdot f =1, \ f^2 = 0, \text{ and } \overline s^2 =  \mathrm{eu} (\overline X). $$
The first two equalities are obvious, and the justification of the third one is as follows: since the section $\overline s$ is everywhere tangent to $\mathcal{F}$ and to the fibers, we get an isomorphism between the tangent bundle and normal bundle to $\overline s$. 
Therefore  the adjunction formula yields
$$\overline s^2  = \mathrm{deg}(N_{\overline s}\rest{\overline{s}})  = - \mathrm{deg}(K_{\overline s}\rest{\overline{s}}) =
 \mathrm{eu} (\overline s) =  \mathrm{eu} (\overline X). $$
From this and \eqref{eq:npoint} we easily deduce that  
\begin{equation} \label{eq:normal bundle} [N_{\overline{\mathcal F}} ] = 2 [\overline{s}] -  \mathrm{eu}(\overline{X} ) [f] .\end{equation}
Now Proposition \ref{p:weaker} asserts that 
$$\chi (\sigma)  = \displaystyle \frac{1}{2\abs{\mathrm{eu}(X)} } ( N_{\overline{\mathcal F}} \cdot \overline{T} + I \cdot \#P). $$
Also, from Proposition \ref{p:cohomological expression of the degree} we have that 
$\delta =\unsur{\vol(X)}  \overline{T} \cdot \overline{s}$, and it is obvious that $\overline{T} \cdot f = 1$. Using the fact that $\mathrm{eu}(\overline X)
=\mathrm{eu}(X)  + \#P$, altogether this yields 
$$\chi (\sigma)   = \frac12  + 2\pi \delta+ (I-1) \frac{\#P}{2\abs{\mathrm{eu}(X)}}.$$ Therefore,
 to finish the proof 
it is enough to show that $I=1$. This is done by considering the particular case of the canonical projective structure 
  induced by the 
uniformization of $X$, since in this case we have that  $\delta = 0$ and $\chi = \frac12$ (see the 
remarks following \cite[Def. 2.1]{Bers1}). The proof is complete. \qed

\section{The   dimension of   harmonic measure: proof of Theorem \ref{theo:dimension}} \label{s:dimension}

In this part we provide the proof of Theorem \ref{theo:dimension}. We start with a result originating 
 in the work of S. Frankel   (see   \cite{frankel}).

\begin{prop} \label{p:frankel}
Let $\varphi$ be the  density of the desintegration of $T$ (resp. $\mu$)
 along the leaves of $\mathcal F $ (see (\ref{eq:foliated})). Then the integral 
\begin{equation} \label{eq:alpha}   A = -  \int _{X_\rho} \Delta_{\mathcal F} \log \varphi \  d\mu \end{equation}
is convergent, and moreover 
\begin{equation} \label{eq:bound on alpha}   0\leq A \leq 1.  \end{equation}
\end{prop} 

Notice that $\varphi$ is defined only up to a multiplicative factor which is constant along the leaves, which shows that $\Delta_{\mathcal F} \log \varphi$ is well defined. The quantity $A$ is called the {\em action} of $T$.

\begin{proof} The positivity and harmonicity of the density $\varphi$ implies  that 
$\Delta_{\mathcal F} \log \varphi = - \norm {\nabla \log \varphi }^2$.  Thus 
by the Harnack inequality for positive harmonic functions we infer that $ \norm {\nabla \log \varphi } $ is uniformly bounded, whence the convergence of the integral in  \eqref{eq:alpha}. To get the bound \eqref{eq:bound on alpha}, we observe that  $\varphi$ 
lifts to the universal cover of $\mu$-a.e. leaf as a positive harmonic function, 
hence the half-plane version of the  Harnack inequality (obtained by taking conjugate harmonic functions and applying the Schwarz-Pick lemma) yields
  $ \norm {\nabla \log \varphi } \leq 1 $. This proves \eqref{eq:bound on alpha}
(see \cite{deroin levi plate} for more details). \end{proof}

The main step of the proof is the following probabilistic  estimate of the measure of a ball inside a fiber. 
For every $ {x} \in M_\sigma$, and $\rho >0$, we denote by   $B_\pu(\tilde{x},\rho)$   
the ball of radius $\rho$ centered at $\tilde{x}$ inside the fiber 
 $\mathbb P^1_{\pi(x)}$. To ease notation, from now on if $x\in M_\sigma$ we denote its  fiber  by $\mathbb{P}^1_x$
  (resp. the corresponding  harmonic measure 
by  $\nu_x$). 

\begin{prop} \label{p:sup dimension} Let $A$ be as in Proposition \ref{p:frankel}.
 For every $\varepsilon >0$, there exists $r_\e >0$ such that for every $0< r < r_\e$,  
\[  \mu \left( x\in M_\sigma ,\   \ \nu_x (B_\pu({x} , r)) \geq r ^{\frac{A}{2\chi}+\varepsilon}) \right) \geq 1 -\varepsilon. \]
\end{prop}

\begin{proof} 
The proof follows an argument of Ledrappier's   \cite[Thm 4.1, p. 372]{ledrappier}, with the difference that the discrete random walk is replaced   by a cocycle   over the ergodic system $(\Omega^{\mathcal{F}} , \sigma , W^{\mathcal{F}}_\mu)$.
 
In view of the next lemma it is useful to recall that the data of a foliated Brownian path is equivalent to that of a Brownian path in $X$
 together with its starting point in the fiber. 

\begin{lem} \label{l:measures of images}
Let $x\in X$ and $C_x$ be a measurable subset of $\mathbb P^1 _x$ such that $\nu_x (C_x) >0$. Then for $W^X_x$-a.e. $\omega: [0,\infty) \rightarrow X$, letting $h_t = h(\omega\rest{[0,t]})$, we have that 
\[   \limsup _{t\rightarrow \infty}  -\frac{\log \nu_{\omega(t)} (h_t (C_x)) }{t}  \leq A. \] 
\end{lem} 

\begin{proof} We first work with foliated Brownian motion. 
Let us introduce the family of  functions $(L_t)_{t\geq 0}$, defined on $\om^{\mathcal{F}}$ by 
\[ L_t (\omega ) = - \log \frac{ (h_t)^{-1} \nu_{\omega(t)} }{\nu_{\omega(0)}} (\omega (0) ) . \]
It is immediate that $L=( L_t  )_{t\geq 0}$ is a cocycle, namely it satifies the relations 
\[  L_{t+s} (\omega) = L_s (\omega) + L_t (\sigma _s (\omega)) , \text{ for every } s,t\geq 0 \text{ and } 
\omega \in \Omega^{\mathcal{F}}. \]
Moreover, in terms of the densities $\varphi$, $L_t$ expresses as
 $ L_t (\omega ) = - \log \frac{\varphi (\omega(t))}{\varphi(\omega(0))}$, hence by   applying the 
  Harnack inequality $\norm{\nabla_\mathcal{F} \log \varphi}\leq 1$  we obtain  the estimate 
\begin{equation}\label{eq:Harnack}  |L_t(\omega) | \leq \text{length} (\omega\rest{[0,t]}), \text{ for every } t\geq 0 \text{ and } \omega \in \Omega^{\mathcal{F}}\end{equation}
(recall that the length in question  here is the homotopic length of $\pi(\omega)$). 
 Thus   the super-exponential decay of the heat kernel on the upper half plane \cite[\S 5.7]{davies}  implies that 
 $L_t$ is $W^{\mathcal F}_\mu$-integrable. 
The subadditive ergodic theorem applied  to the cocycle $L$ and   the ergodic system $(\Omega ^{\mathcal{F}}, \sigma , W^{\mathcal{F}}_\mu)$  shows that  $\frac{L_t(\omega)} {t}$ converges a.s. 
to a limit independent of $\omega$. Arguing exactly as in 
Lemma \ref{l:computation lyapunov exponent} shows that this limit equals $A$, that is, 
\begin{equation} \label{eq:cocycle L} \frac{L_t(\omega)} {t} \underset{t\rightarrow \infty}\longrightarrow   A, 
\text{ for $W^{\mathcal{F}}_\mu$-a.e. } \omega \in \Omega^{\mathcal{F}}. \end{equation}
From this, we infer that for a.e. $x\in M_\sigma$, and $W_x^{\mathcal{F}}$ a.e. $\omega$, 
 if $C_x \subset \mathbb P^1 _x$ is a measurable subset such that $\nu_x(C_x ) >0$, then   
\begin{equation}\label{eq:inversion limit}  \lim _{t\rightarrow \infty} \frac{1}{\nu_x(C_x )} \int _{C_x } -\frac{1}{t} \log \left( \frac{ (h_t)^{-1} \nu_{\omega(t)} }{\nu_{\omega(0)}}  (\omega(0) ) \right) d\nu_{\omega(0)}  = A .\end{equation}
This follows from \eqref{eq:cocycle L}, Fubini's theorem and 
   the dominated convergence theorem, since for a generic $\omega$, 
  $\text{length} (\omega\rest{[0,t]})= O(t)$ so by \eqref{eq:Harnack} the argument of the integral 
 in \eqref{eq:inversion limit} is bounded independently of $t$. 
  We now use  the convexity of the function $-\log$  which  by Jensen's inequality implies that 
\[  \frac{1}{\nu_x(C_x )} \int _{C_x } -\frac{1}{t} \log \left( 
\frac{ (h_t)^{-1} \nu_{\omega(t)} }{\nu_{\omega(0)}}(\omega(0) )\right)  d\nu_{\omega(0)} \geq 
- \frac{1}{t} \log \left( \frac{\nu_{\omega(t)} (h_t (C_x) ) }{\nu_{\omega(0)} (C_x) } \right) , \]
and so we deduce that for a.e. $x\in X$, as soon as $\nu_x(C_x) >0$, we have that 
\[   \limsup _{t\rightarrow \infty} - \frac{1}{t} \log \nu_{\omega(t)} (h_t (C_x) ) \leq A. \]

 Notice that this property makes no reference to the starting point in the fiber,  so it can be stated as well for 
 a.e. $x\in X$ and $W^X_x$ a.e. $\omega\in \om^X_x$

To finish the proof it remains to  see that this statement holds for  \textit{every} $x$. 
For this, we first  observe that  if $C_x$ has positive measure, then 
 for $W^X_x$-a.e. $\omega$,   $\nu_{\omega(1)} ( h_1 (C_x)) >0$. Furthermore, 
  the distribution of $\omega(1)$ is absolutely continuous, 
 so that  the previous estimates hold when $x$ is replaced by $\omega(1)$. The assertion then follows from  
 the Markov property of Brownian motion. 
\end{proof}

Let us  resume the proof of Proposition \ref{p:sup dimension}. Fix $\e>0$, and put $\eta = \frac{4\chi+A}{4\chi^2} \e$. 
By Lemma \ref{l:technical step}, for $W^{\mathcal F}_{\mu}$-a.e. $\omega$,    the convergence 
 $\frac{1}{t} \log \norm{Dh_t (y)}\underset{t\rightarrow\infty}\longrightarrow  \lambda$, 
 holds  uniformly on compact subsets of $\mathbb P_{\omega(0)} ^1 \setminus \set{r(\omega)}$. 
  So if we let $R = \frac{1}{2}d_\pu (\omega(0), r(\omega))$,   there exists $t_1 = t_1 (\omega, \varepsilon)$ such that if 
  $ t\geq t_1$, 
\[   h_t ( B(\omega(0), R ))\subset  B( \omega(t), e^{(\lambda + \eta)t} ) . \] 
On the other hand by the previous lemma,  for $W_\mu^{\mathcal{F}}$ 
a.e. $\omega$, there exists $t_2 = t_2 (\omega, \varepsilon)$ such that for $t\geq t_2$, 
\[  \nu_{\omega(t)} ( h_t( B(\omega(0), \eta ) ) \geq  e^{ -(A +\eta) t} .\]
So  we infer that for $t\geq \max(t_1, t_2)$,  
\begin{equation}\label{eq:measures of images}   \nu_{\omega(t)} (B( \omega(t), e^{(\lambda + \eta)t} ) ) \geq  e^{ -(A +\eta) t} .\end{equation}
For every $t>0$ let $\Omega_t$ be the set of paths $\omega\in \Omega$ such that $\max (t_1,t_2) \leq t$. Clearly $W^{\mathcal{F}}_\mu(\Omega_t)\geq 1-\e$    for $t\geq t(\e)$.   
Setting $r = e^{ (\lambda + \eta) t} $, if $\omega\in \om_t$ and $x=\omega(t)$,  for $t\geq t(\e)$ 
we have that 
\[  \nu_{ {x} } (B ( {x} , r ) ) \geq r ^{-\frac{A + \eta}{\lambda + \eta}}  = r^{\frac{A + \eta}{2\chi - \eta}} . \]    

This finishes the proof  since the image of $W^{\mathcal{F}}_{\mu}$ under 
$\omega \mapsto \omega(t)$ is the measure $\mu$, and 
$\eta$ was chosen so that     $\frac{A + \eta}{2\chi - \eta}<\frac{A}{2\chi}+\e$.
\end{proof} 

An estimate  similar to that of Proposition \ref{p:sup dimension} holds in every fiber. 

\begin{cor} \label{c:volume balls}
Let $x\in X$ and $\varepsilon >0$. There exists $r_\e>0$ such that if $0<r <r_\e$ then 
\[  \nu_x \left( \tilde x \in \mathbb P _x^1 \ :\ \nu_x ( B_\pu(\tilde x , r ) ) \geq r ^{\frac{A}{2\chi}+2\varepsilon} \right) \geq 1-\varepsilon. \]
\end{cor}

\begin{proof} This is due to the fact that the holonomy map $h_\gamma$ corresponding to a path $\gamma: [0,1]\rightarrow X$ of 
length $\ell$ is bilipschitz with constant depending only on $\ell$, and moreover it sends the measure $\nu_x$ to a measure absolutely 
continuous with respect to $\nu_y$, whose density is bounded from above and below by positive 
constants depending only on $\ell$. Now if $\ell$ is fixed,  the proportion of points in a given fiber lying at leafwise distance $\ell$ from a point   satisfying the conclusion of Proposition  \ref{p:sup dimension} tends to 1 when $\e\cv0$, so we are done. 
 \end{proof}

We are now ready to finish the proof of Theorem \ref{theo:dimension}. Fix a real number $s> \frac{A}{2\chi}$ and  $\e$ such that 
$0<2 \e < s-  \frac{A}{2\chi}$. With $r_\e$ as in Corollary \ref{c:volume balls}, for every $r<r_\e$, 
consider the set $$E_{r,\e} = \set{\tilde x\ \in \mathbb{P}^1_x ,  \ \nu_x ( B(\tilde{x},r) \geq r ^{\frac{A}{2\chi} + 2\varepsilon}}.$$
From  Corollary \ref{c:volume balls} we know that $\nu_x (E_{r/5 ,\e})> 1-\e$. Furthermore, 
a classical covering argument gives an estimate of the $s$-dimensional Hausdorff measure of $E_{r/5,\e}$.  Indeed by the Vitali covering lemma there exists a covering of $E_{r/5, \e}$ by balls $B_\pu(\tilde x_i, r)$ centered on $E_{r/5, \e}$ and
 of radius $r$ such that the corresponding balls of radius $\frac{r}{5}$ are disjoint. This disjointness together with the measure estimate 
 imply that this set of balls has cardinality at most $$N \leq \lrpar{\frac{r}{5}}^{-\lrpar{\frac{A}{2\chi}+2\e}}.$$ Therefore, 
\begin{equation}\label{eq:hs}
\mathcal{H}_s(E_{r/5 ,\e}) \leq \sum_i (2r)^s \leq 2^s 5^{\frac{A}{2\chi}+2\e} r^{s- \frac{A}{2\chi}-2\e}.
\end{equation}
 
 We now set  $\e_n  =2^{-n}$, $r_n = r(\e_n)/5$ , and put $F_k = \bigcap_{n\geq k} E_{r_n,\e_n}$. Since for every $n \geq k$, 
 $F_k \subset E_{r_n, \e_n}$, from \eqref{eq:hs} we infer that $\mathcal{H}_s(F_k) = 0$. On the other hand 
 $\nu_x(F_k)\geq 1-\lrpar{\unsur{2}}^{k-1}$, so if we let $F = \bigcup_{k\geq 1} F_k$ we have that $\nu_x(F)=1$ and $\mathcal{H}_s(F) =0$, hence  $\mathrm{dim}_H(\nu_x)\leq s$. Since $s> \frac{A}{2\chi}$ was arbitrary, we conclude that 
  $\mathrm{dim}_H(\nu_x)\leq \frac{A}{2\chi}\leq \unsur{2\chi}$, as asserted. 
  
  In particular, it follows from Theorem \ref{theo:formula} that $\mathrm{dim}_H(\nu_x)\leq 1$, 
  and if equality holds then  $\delta=0$ and $A = 1$. 
Lemma \ref{lem:zero degree} below shows that if $\mathrm{dim}_H(\nu_x) =  1$, then $\sigma\in \overline{B(X)}$. 
Conversely, if $\sigma\in \overline{B(X)}$, it follows from  Makarov's celebrated theorem \cite{makarov} that the harmonic measures are supported by a set of dimension 1 (the measures $\nu_x$ coincide with the classical harmonic measure in this case, 
see the next lemma). This completes the proof of the theorem. \qed
  
\begin{lem}\label{lem:zero degree}
Let $\sigma$ be  a parabolic projective structure with $\deg(\sigma) = 0$. Let $(m_y)_{y\in \dev(\widetilde X)}$ 
be the usual harmonic measure of 
the open set $\dev(\widetilde X)\subset \pu$.  
Then for every $x\in \widetilde X$, $\nu_x = m_{\dev(x)}$. 

In addition the action $A$ equals $1$ if and only if $\dev$ is injective, that is, $\sigma\in \overline{B(X)}$ 
(see the discussion on the density theorem in \S \ref{subs:Teichmuller}). 
\end{lem}

\begin{proof} The first part is proved by using the  conformal invariance of Brownian motion. Indeed, any Brownian path $\eta:[0,\infty[$ 
(relative to the  spherical metric on $\mathbb P^1$, say) starting at $y$ hits a.s. the boundary of $\dev(\widetilde{X})$ at a first moment 
$S>0$. We denote by $p= \eta (S)$. The distribution of $p$ is by definition the harmonic measure $m_y$. The path $\eta\rest{[0,S)}$ can 
be lifted to a continuous path $\widetilde{\eta} : [0,S) \rightarrow \widetilde{X}$ starting at $\widetilde{\eta}(0)= x$, and satisfying  $\dev \circ 
\widetilde{\eta} =\eta$. Let  $\omega :[0,T)\rightarrow \widetilde{X}$ be the reparametrization of $\widetilde{\eta}$ defined by $ \omega (t) = 
\widetilde{\eta}(s) $, where
\begin{equation} \label{eq:conformal invariance} t = \int _0 ^s \norm{D \dev^{-1}(\eta(u)) }^2 du. \end{equation}
Here $\dev^{-1}$ is understood as the analytic continuation along $\eta$ of the inverse of $\dev$ defined at the neighborhood of $y$ and 
such that  $\dev^{-1} (y)=x$. The conformal invariance can be stated in the following form: $\omega$ is a model for a Brownian path 
starting at $x$ for the Poincar\'e metric  (see e.g.  \cite[Section 1]{carne}). Since $\omega$ tends to infinity in $\widetilde{X}$ when $t$ 
tends to $T$, we see that $T=+ \infty$ a.s. Moreover,   a.s. $\lim_{t\rightarrow +\infty} \dev(\omega(t)) = p$, which implies using Definition-
Proposition \ref{defprop:harmonic measure} that $m_y= \nu_x$.

\medskip

Let now address  the second part of the lemma. We first prove that $A < 1$ if $\dev$ is not injective. We  need the concept of an extremal positive harmonic function on the universal covering of $\widetilde{X}$. Such a function is (by definition) the composition of a biholomorphism from $\widetilde{X}$ to $\mathbb H$ with  the imaginary part function $\Im : \mathbb H \rightarrow (0,\infty)$. It will be important to notice that the subgroup of $\text{Aut}(\widetilde{X})$ that preserves an extremal positive function is abelian (this is the group of translations in the coordinate where the function is the imaginary part). 
The following statement is a consequence of the case of equality in the Schwarz-Pick lemma: 

\textit{A function $\varphi : \widetilde{X} \rightarrow (0,\infty)$ is an extremal positive harmonic function if and only if at some (and hence  all) point $x\in \widetilde{X}$  one has $\norm{\nabla \log \varphi (x)} = 1$. }

Now assume that $\dev$ is not injective. In such a situation, the covering group $\text{ker} (\hol)$ is a non trivial normal subgroup of $\pi_1(X)$. Recall  (item (ii) of Proposition \ref{p:harmonic measures}) 
that the family of harmonic measures $\{ \nu_x \}_{x\in \widetilde{X}} $ satisfies the equivariance relation $\nu_{\gamma x} = \hol (\gamma) _* \nu_x$ for every $x\in \widetilde{X} $ and every $\gamma \in \pi_1(X)$. Hence, the density of the disintegration of $\widetilde{T}$ along 
the leaves is a function $\varphi  : \widetilde{X}\times \mathbb P^1 \rightarrow (0,\infty)$ which belongs to $L^1_{loc} (\text{vol} \otimes \nu)
$  and is invariant under the group $\text{ker} (\hol)$. 
This subgroup being non trivial and normal, its limit set as a subgroup of isometries of $\text{Aut}(\widetilde{X})$ for the 
Poincar\'e metric is the whole $\partial \widetilde{X}$. 
In particular, it contains non abelian free subgroups \cite{beardon}.
As a consequence, the density  $\varphi(\cdot, z)$ of the disintegration  of $T$ cannot be extremal. In 
particular, $\norm{\nabla_{\mathcal F} \log \varphi }<1$ a.s. This proves that $A < 1$, as required. 

It remains to prove that if $\dev$ is injective, then $A = 1$. In this case, the holonomy representation is injective with image a discrete 
subgroup of $\text{PSL} (2,\mathbb C)$. In particular, using Remark \ref{r:measurable conjugacy} we see that 
 the foliated bundle $(M_\sigma,\mathcal F_\sigma, T_\sigma)$ is 
 measurably conjugate to the bundle $(M_{\sigma_{\rm Fuchs}},\mathcal F_{\sigma_{\rm Fuchs}}, T_{\sigma_{\rm Fuchs}})$ where $\sigma_{\rm Fuchs}$ is the uniformizing structure on $X$. Hence $A =A_{{\rm Fuchs}}$. 
 But the densities of the disintegration of $T_{\rm Fuchs}$ along the leaves are
 given by the Poisson kernel in the uniformization coordinates, in particular these are extremal positive harmonic functions. 
 We conclude that $A = A_{\rm Fuchs} = 1$, and the proof of the lemma is complete. 
\end{proof}


\begin{rmk}\label{rmk:makarov}
Let $\om_0 \subset \pu$ be any component of the discontinuity set of a finitely generated 
Kleinian group $\Gamma$. Using Theorem \ref{theo:dimension} 
 we can recover the classical Jones-Wolff theorem \cite{jones wolff} that the dimension of the harmonic measure of $\partial \om_0$ is bounded by $1$, with strict inequality unless $\om$ is simply connected (see \cite{puz} for another dynamical proof of this fact). 
Indeed, from Lemma \ref{lem:zero degree}, it is enough to show   there exists a Riemann 
surface $X$ of finite type and a parabolic projective structure on $X$ with zero degree such that $\om_0 = \dev(\widetilde X)$.
This simply follows from the Ahlfors finiteness theorem:   first take a finite index torsion free subgroup $\Gamma' \subset \Gamma$, and 
define $X:= \Gamma_0' \backslash \om_0$ where $\Gamma_0'\subset \Gamma'$ is the stabilizer of $\om_0$. 
 \end{rmk}

\section{Applications to Teichm\"uller theory}\label{sec:teich}

\subsection{Preliminaries}\label{subs:Teichmuller}
All this is well-known, but not so easy to locate in the literature when $X$ is non-compact. 


Recall that $X$ is assumed to be a Riemann surface of finite type, that is biholomorphic to 
$\overline{X} \setminus P$ where $\overline{X}$ is compact and $P$  is a finite set of punctures.
Introduce a projective structure $\overline{\sigma}$  on $\overline{X}$, which can  always can be done e.g. by uniformization.
 For every projective structure $\sigma$ on $X$, consider the holomorphic 
quadratic differential on $X$ defined by 
\[   q  = \{  w , x \} dx^2 \]
where $x$ and $w$ are projective coordinates for the projective structures $\overline{\sigma}$ and $\sigma$ respectively, 
and as usual $\{ w, x \}$ is the Schwarzian derivative. By the cocycle property of  the Schwarzian, we infer that 
 $\{ w,x\}dx^2 = \{ z, x\} dx^2 + \{ w,z \} dz^2  $. Remembering  that $\{ w,x\}$ vanishes if $w$ is a Moebius transformation, we see that 
 the differential $q$ is well-defined, that is 
  does not depend on the chosen coordinates $x$ and $w$. 
  Moreover, a result due to Fuchs and Schwarz shows that a projective structure on $X$ is parabolic if and only if the Laurent series 
  expansion of $q$ at the neighborhood of every point of $P$ takes the form $q(x) = \big(\frac{1}{2x^2} + \text{l.o.t.}\big)dx^2$, see~\cite[Th\'eor\`eme 10.1.1, p. 291]{unif}. Hence, the space of parabolic projective structures on $X$ is an complex affine
  space directed by the set of meromorphic quadratic differentials having poles on $P$ of order at most $1$. 
  This space is the set of sections of the line bundle $L = 2K + O(P)$, where $K$ is the canonical divisor of $\overline{X}$. 
  By Riemann-Roch 
\[  h^0 (\overline{X}, L) - h^0 (\overline{X}, K - L ) = \mathrm{deg} (L) + 1 - g .\]
Since $K-L= -K-O(P)$ has no non trivial sections, and that $\mathrm{deg} (L) = 4g - 4 + |P|$, we deduce 
\[ h^0 (\overline{X}, L) = 3g-3 + |P| .\]
Thus the set of parabolic projective structures on $X$ is a complex affine space of dimension $3g- 3 + n$, where $n$ is the number of 
punctures. Observe that for the once punctured torus  or the fourth punctured   sphere, the  dimension equals $1$.

\medskip

We denote by $ T_{g,n}$ the Teichm\"uller space of equivalence classes of marked Riemann surfaces biholomorphic to a compact Riemann surface of genus $g$ punctured at $n$ distinct points. Here a marking of the Riemann surface $Y$ will refer to the data of a universal covering $\widetilde{Y} \rightarrow Y$ together with an identification of the covering group $\pi_1(Y)$ of this covering with $\pi_1(X)$. Two marked surfaces are considered as equivalent if there exists an equivariant holomorphic diffeomorphism between the universal covers. 

Let $Y \in \mathcal T_{g,n}$. Denote by $c(Y)$ the complex conjugation of $Y$, keeping the marking fixed. The Bers simultaneous
 uniformization theorem \cite[Theorem 1]{Bers} asserts that there exists a faithful discrete representation $\rho _{X,Y} : \pi_1 (X) \rightarrow \text{PSL} (2,\mathbb C)$, uniquely defined up to conjugation, such that
  the Riemann sphere admits  a $\rho$-invariant partition of the form  
\begin{equation}\label{eq:quasi-Fuchsian} \mathbb P^1  = D_X\cup \Lambda \cup D_Y,\end{equation} 
where $D_X$ and $D_Y$ are two simply connected domains, $\Lambda$ is a topological circle, and  such that the marked Riemann surfaces $\rho (\pi_1(X)) \backslash D_X$ and $\rho (\pi_1(X)) \backslash D_Y$ are respectively equivalent to $X$ and $c(Y)$. 
A representation with an invariant decomposition such as \eqref{eq:quasi-Fuchsian} is called quasi-Fuchsian. Thus an element $Y\in \mathcal T_{g,n}$ produces a parabolic $\mathbb P^1$-structure $b(Y)\in P(X)$ on $X$, defined as the $\rho$-equivariant identification between $\widetilde{X} $ and $D_X$. 

It turns out  that the map $Y\in T_{g,n} \mapsto b(Y) \in P(X)$  
 is a holomorphic embedding onto a bounded open subset $B(X)\subset  P(X)$, 
 known as the {\em Bers embedding} (or {\em Bers slice}) of  $T_{g,n}$. 
The  holomorphicity of $b$ follows from the  holomorphic dependence of the solution of the Beltrami equation with respect to parameters.
The boundedness of $B(X)$ follows from Nehari's estimate 
for the Schwarzian of univalent meromorphic functions defined on the hyperbolic disc, and its openness from 
the so-called Ahlfors-Weill 
extension lemma (see e.g. \cite{gardiner lakic}). 
 
 \medskip
 
Due to deep recent advances in Kleinian group and 3-manifold theory, there is a now
 good understanding of the structure of $\overline{B(X)}$. To 
be precise, the {\em density theorem} (formerly known as the Bers density conjecture), specialized to our context, asserts that 
$\sigma\in \overline B(X)$  if and only if $\dev_\sigma$ is injective. 
This means in particular that the image of $\dev$ is a simply connected component of the discontinuity set of $\hol_\sigma(\pi_1(X))$, which 
uniformizes $X$.  An equivalent formulation 
is that $\sigma\in \overline{B(X)}$ if and only if $\deg(\sigma)=0$ and  $\hol_\sigma $ is faithful. 
When $X$ is compact, 
this was explicitly proved by Bromberg \cite{bromberg annals}. In the general case, this statement is generally accepted by the experts
 as being a 
consequence of the ending lamination theorem of Minsky \cite{minsky} and Brock, Canary and Minsky \cite{bcm}
(see Ohshika \cite{ohshika} and Namazi-Souto \cite{namazi souto} for the derivation of the density theorem from the ending lamination 
theorem in the whole character variety). It seems, however, that no detailed proof of this fact has appeared yet.

\subsection{Holomorphic convexity of Bers slices}
Here we prove Theorem \ref{theo:convex}.  It is known that every component of the interior 
of a polynomially convex set is polynomially convex  (i.e. a Runge domain)
 (see e.g. \cite[Prop. 2.7]{fs2}), but the converse is false (this fails e.g. for $U = D(0, 2)\setminus D(1,1) \subset \cc$). 
In particular the second statement of the theorem (the polynomial convexity of 
$B(X)$) follows from the first (the polynomial convexity of  $\overline{B(X)}$). 

Actually, one may derive the polynomial 
convexity of $B(X)$ from a simple, direct argument. Indeed, consider in $X$  the set $\set{\deg = 0}$ of projective structures with vanishing degree. Equivalently, by Proposition \ref{p:vanishing} such a structure is of quotient type.
 It was shown in \cite{kra maskit} that 
$\set{\deg = 0}$ is a compact subset of $P(X)$. Since in $\cc^n$ convexity with respect to polynomials and psh functions coincide 
\cite[Thm 1.3.11]{stout}, we infer that    $\set{\deg=0}$  is polynomially convex. Now it is a result due to
 Shiga and Tanigawa \cite{shiga tanigawa} and Matsuzaki \cite{matsuzaki} that  
  the interior  in $P(X)$ of the set of projective structures with  discrete holonomy is the set of projective structures with quasifuchsian holonomy. Thus $\Int \set{\deg = 0} = B(X)$, and the polynomial convexity of $B(X)$ follows.  
 
 \medskip

We now turn to the polynomial convexity of $\overline{B(X)}$. 
A connected component of a polynomially convex set is polynomially convex 
 \cite[Cor. 1.5.5]{stout}, so it is enough to show that $\overline{B(X)}$  is a connected component of $\set{\deg = 0}$. This will be based on the following amusing lemma. 

\begin{lem}\label{lem:connex}
Let $(\rho_{\la})_{\la\in \La}$ be a holomorphic family of   representations of a finitely generated group $G$ into $\PSL$, parameterized by some   complex manifold $\La$. If $K\subset \La$ 
is a compact connected set of discrete and non-elementary representations then $\ker (\rho_\la)$ 
is constant over $K$. In  particular if $K$ contains a faithful representation, then all representations in $K$ are faithful. 
\end{lem}

Admitting this result for the moment, let us finish the proof.
 Let $K$ be the connected component of $\set{\deg = 0}$ 
 containing $\overline{B(X)}$.  By Proposition \ref{p:vanishing} for every $\sigma\in K$, $\mathsf{hol}_\sigma$ is discrete.  
Applying  Lemma \ref{lem:connex}, we infer that all representations in $K$ are faithful. 
By the density theorem (see the end of \S\ref{subs:Teichmuller}), 
  the inclusion  $K\subset \overline{B(X)}$ holds, hence  $K = \overline{B(X)}$, and the proof is complete.
\qed

\begin{proof}[Proof of Lemma \ref{lem:connex}] Without loss of generality we may assume that $\La$ is connected. 
Consider the set of subgroups $\ker (\rho_\la)$ as $\la$ ranges in $\La$. 
Our first claim is that  this set  is at most countable.
The argument is based on   basic finiteness (Noetherian) properties in analytic geometry. Let $\lo\in \La$ and put 
$K_0 = \ker (\rho_\lo)$.  Define
 $$Z'(K_0) = \set{\la\in \La,\ \forall g\in K_0, \rho_\la(g)  = \mathrm{id}},$$ and let $Z(K_0)$ be the component of $Z'(K_0)$ containing $\lo$.
  If $\la\in Z(K_0)$, $\ker(\rho_\la)  \supset K_0$
nevertheless  equality needn't hold.  On the other hand we observe that if $\la$ is a generic point in $Z(K_0)$, that is, chosen outside a countable family of  proper analytic subvarieties,  then $\ker(\rho_\la)  = K_0$. Indeed for every $g\in G\setminus K_0$, the set   
$\set{\la\in Z(K_0), \ \rho_\la(g) = \mathrm{id}}$ is a proper subvariety of $Z(K_0)$, since it does not contain $\lo$.  

Conversely, a similar argument shows that if $V\subset \La$ is any irreducible variety, the subgroup 
$$K(V)  = \set{g\in G, \forall \la\in V, \rho_\la(g) =\mathrm{id}}$$ is the   kernel of generic representations in $V$. 

So if $K_0$ is as above, $Z'(K_0)$ has at most countably many irreducible components, each of which associated 
with a generic kernel (for $Z(K_0)$, this  is precisely $K_0$).  Now we observe that 
locally $Z'(K_0)$ is defined by finitely many equations, that is 
 there exists a finite number of elements $g_i\in G$, $i=1\ldots N$,  such that  for $\La'\Subset \La$, 
 $$Z'(K_0) \cap \La' = \set{\la\in \La',\ \forall i =1\ldots N, \rho_\la(g_i)  = \mathrm{id}}.$$
This leaves only countably many possibilities for the generic kernels, and our claim is proved. 

\medskip

Under the assumptions of the lemma, label all   kernels of representations in $K$ as $(H_i)_{i\in \nn}$ and write accordingly $K$ as a disjoint union $K = \bigcup K_i$, where $K = \set{\la,\ \ker (\rho_\la)= H_i}$. The next claim is that for every $i$, $K_i$ is closed. For this we use the precise version of the Chuckrow (Margulis-Zassenhaus-Jorgensen) theorem stated in \cite[Thm 8.4 p.170]{kapovich}: if $\rho_p$ is  a sequence of discrete faithful representations of some non-radical group $\Gamma$ into $\PSL$, algebraically converging to some representation $\rho$ of $\Gamma$, then  $\rho$  is also discrete and faithful. Recall that a group $\Gamma$ is said non-radical if it does not admit infinite normal nilpotent subgroups. A non-elementary subgroup of $\PSL$ contains rank $2$ non abelian free subgroups and in particular is non-radical. 
 
Let now $(\la_p)\in K_i^\nn$ be   sequence converging to some  $\la \in K$.  For every $p$, 
$\rho_{\la_p}$ is a discrete faithful non-elementary representation of $G/H_i$, which is therefore non-radical. It is clear that $H_i\subset \ker(\rho_\la)$ so $\rho_\la$ can be viewed as a representation of $G/H_i$. Hence by Chuckrow's theorem $\ker(\rho_\la) = H_i$, and we conclude that $K_i$ is closed. 

\medskip

 We have thus written $K$ as an at most countable union of  disjoint  closed sets $K_i$. A (not so well-known!) theorem of Sierpi\'nski \cite{sierpinski} asserts that such a decomposition must be trivial. The result is proved.
\end{proof}

\begin{rmk}
If $\dim(P(X)) =1$, 
it is not necessary to use the density theorem. Indeed
in dimension 1, polynomial convexity  simply means that $K^c$ has no bounded component 
(there is no such simple topological characterization in higher dimension).  
 What our argument says is that 
$\overline{B(X)}$ is contained in a polynomially convex set $K$ with $\mathring{K} = B(X)$. This
 directly implies that $\overline{B(X)}^c$ has no bounded components (this is left as an exercise to the reader).
 \end{rmk}

\subsection{Exterior powers of the bifurcation current}

\begin{proof}[Proof of Theorem \ref{theo:tbif}]
The argument is based on the fact that an isolated minimum of the continuous psh function $\delta$ must belong to $\supp(dd^c\delta)^{3g-3}$. 
This is a consequence of  the so-called {\em comparison principle} for psh functions \cite[Thm A]{bedford taylor}.
A similar  idea appears in  the work of Bassanelli and Berteloot (see \cite[Prop. 6.3]{basber}).  

It is a theorem due to Hejhal \cite[Thm 6]{hejhal schottky} that covering projective structures with Schottky holonomy are isolated points of $\set{\delta=0}$. Therefore,  for such a $\sigma_0$,  we infer that $\sigma_0 \in \supp(\tbif^{3g-3})$. Finally, we use a result 
due to Otal \cite{otal} (see also Ito \cite{ito}): when $X$ is compact, $\fr B(X)$ is contained in the accumulation set of projective structures with degree 0 and Schottky holonomy.   We conclude that 
$\fr B(X)\subset \supp(\tbif^{3g-3})$. 
\end{proof} 

Corollary \ref{coro:deterministic} is an immediate consequence of the following equidistribution result 
 in the spirit of \cite[Thm C]{Bers1}. 
If $\gamma$ is a closed geodesic on $X$ we let 
$$Z(\gamma, t) = \set{\sigma\in P(X), \ \tr^2(\mathsf{hol}_\sigma) = t}$$ 
(notice that since $\mathsf{hol}_\sigma$ is well-defined up to conjugacy, so it makes sense to speak of its trace). 
Fix a sequence $(r_n)_{n\geq 1}$ 
such that for every $c>0$, the series $\sum e^{-cr_n}$ converges.
The notion of a random sequence of geodesics of length at most $r_n$ was discussed at length in \cite{Bers1}.
We say that a holomorphic family of representations 
$(\rho_\la)_{\la\in \La}$ is {\em reduced}  if the associated mapping
$\La\cv\mathcal X(\pi_1(X), \PSL)$ to the character variety has discrete fibers. Notice that since 
$\mathsf{hol}:P(X)\cv \mathcal X(\pi_1(X), \PSL)$ is injective, this property is satisfied in the context of Corollary \ref{coro:deterministic}. 

\begin{prop}
Let $X$ be a hyperbolic Riemann surface of finite type and $(\rho_{\la})_{\la\in \La}$ be a reduced 
holomorphic family of non-elementary representations of $\pi_1(X)$, with $\dim (\La)\geq k$. 
Let $(r_n)$ be as above, and fix $t\in \cc$.
 For $i= 1,\ldots , k$ fix a   sequence $(\gamma_n^i)$ of independent random 
closed geodesics of length at most $r_n$. Then almost surely, 
\begin{equation}\label{eq:tbifk}
\lim_{n_1\cv\infty} \cdots \lim_{n_k\cv\infty} 
\unsur{4^k\prod_{i=1}^k \mathrm{length}(\gamma_{n_i}^i)} \big[Z(\gamma_{n_1}^1, t)\cap \cdots \cap Z(\gamma_{n_k}^k, t)\big] 
= T_{\rm bif}^k,
\end{equation}
\end{prop}
  Note that in \eqref{eq:tbifk} the intersections are counted with multiplicity (i.e. in the sense of holomorphic chains, see \cite[Chap. 12]{chirka}).

\begin{proof} The proof is similar to that of \cite[Thm 6.16]{preper}. We argue by induction on $k$. For $k=1$ this is \cite[Thm C]{Bers1}. Now assume that the result has been proved for $k$. Since $T_{\rm bif}$ has continuous potential, it follows from 
\eqref{eq:tbifk} that 
$$\unsur{4^k\prod_{i=1}^k  \mathrm{length}(\gamma_{n_i}^i)} \big[Z(\gamma_{n_1}^1, t)\cap \cdots \cap Z(\gamma_{n_k}^k, t)\big] \wedge 
\tbif$$ converges to $\tbif^{k+1}$ as $n_k, \ldots , n_1\cv\infty$ successively. Now, when $n_1, \ldots , n_k$ are large and fixed, we apply 
\cite[Thm C]{Bers1} to the family $(\rho_{\la})_{\la\in Z(\gamma_{n_1}^1, t)\cap \cdots \cap Z(\gamma_{n_k}^k, t)}$  (the reducedness assumption is used here to ensure that this family is non constant) to get that 
\begin{align*}\lim_{n_{k+1}\cv\infty} 
\unsur{4^{k+1}\prod_{i=1}^{k+1} \mathrm{length}(\gamma_{n_i}^i)} & \big[Z(\gamma_{n_1}^1, t)\cap \cdots \cap Z(\gamma_{n_{k+1}}^{k+1}, t)\big] \\ &= \unsur{4^k\prod_{i=1}^k \mathrm{length}(\gamma_{n_i}^i)} \big[Z(\gamma_{n_1}^1, t)\cap \cdots \cap Z(\gamma_{n_k}^k, t)\big] \wedge \tbif \end{align*}
and we are done. \end{proof}

\appendix

\section{Branched projective stuctures}\label{sec:appendix} 

A branched $\mathbb P^1$-structure on $X$ 
 is by definition an equivalence class of development-holono\-my pairs $(\dev, \hol)$, where  $\hol$ is a representation of $\pi_1(X) $ with 
 values in $\text{PSL}(2,\mathbb C)$ and $\dev : \widetilde{X} \rightarrow \mathbb P^1$ is a $\hol$-equivariant (non constant) meromorphic 
 map. Two development-holonomy pairs are considered as equivalent if they are of the form $(\mathsf{dev},\mathsf{hol})$ and $(A\circ 
 \mathsf{dev},A \circ \mathsf{hol}\circ A^{-1})$ for some $A\in \PSL$. Thus the only difference with the classical (unbranched)
  case is that the developing 
 maps are allowed to have critical points. These points are organized as a finite number of orbits under $\pi_1(X)$, 
 and their projections in $X$ are called the branched points. 
 If $X$ is a punctured Riemann surface, we can define a notion of branched parabolic $\mathbb P^1$-structure   exactly as before,
 by specifying it to be induced by $\log z$ near  the cusps. 

\medskip 

Examples of branched $\mathbb P^1$-structures come from conformal metrics with constant curvature $-1,0,1$ on $X$ and conical angles 
multiple of $2\pi$. In particular, a non constant meromorphic map from $X$ to $\mathbb P^1$ defines a branched $\mathbb P^1$-structure 
with trivial holonomy. Quadratic differentials are other kind of examples associated with a flat metric. 
We refer more generally to \cite{veech}. Those examples of non negative curvature have elementary holonomies. Nevertheless, the 
holonomy of a branched $\mathbb P^1$-structure induced by a conformal metric of curvature $-1$ and conical angle multiple of $2\pi$ is 
always non elementary. We refer to \cite{troyanov} for the construction of many examples. 

When $X$ is compact, Gallo, Kapovich and Marden \cite{gkm} showed that if $\rho: \pi_1(X)\rightarrow \text{PSL}(2,\mathbb C)$ is a non 
elementary representation which does not lift to $\text{SL}(2,\mathbb C)$, then $\rho$ is the holonomy  of a branched  $\mathbb P^1$-
structure with exactly one branch point of angle $4\pi$ (for a certain Riemann surface structure on $X$ which depends on $\rho$). 
On the other hand $\rho$ is not the holonomy of a unbranched $\mathbb P^1$-structure. 

\medskip

Some of our results extend 
{\em mutatis mutandis}
to  branched projective structures with non elementary holonomy. 
For instance, the degree of a branched $\mathbb P^1$-structure is defined exactly as in Definition-Proposition \ref{defprop:degree}, 
the proof being identical to the unbranched case. The Lyapunov exponent depends only on the Riemann surface structure and on the 
holonomy representation, hence it has already been defined in our previous work \cite{Bers1}. 
Finally, Hussenot's   definition of the harmonic measures was  actually
 introduced  in the context of branched $\mathbb P^1$-structures with non elementary holonomy. 

\medskip

In this appendix we indicate how our formula relating the Lyapunov exponent to the degree needs to be modified in the branched case. 

\begin{theo} \label{theo:formula branched}
Let $\sigma$ be a parabolic branched $\mathbb{P}^1$ structure on a hyperbolic 
 Riemann surface $X$ of finite type. Let $k$ denote the number of branched points, counted with multiplicity.
Then, with notation as in Theorem \ref{theo:formula},  the following formula  holds:
$$
 \displaystyle \chi (\sigma) = \frac{1}{2} + 2\pi\delta(\sigma)  - \frac{k}{\abs{\mathrm{eu(X)}}}  = \frac12+ \frac{\deg(\sigma) - k }{\abs{\mathrm{eu(X)}}}
$$
\end{theo}

\begin{proof} [Sketch of proof]
Introduce as in the unbranched case 
the flat bundle $(\overline{M_\sigma} , \overline{\mathcal F_\sigma})$, the holomorphic section $\overline{s}$ (the compactification of the 
graph of the developing map at the level of the universal cover) and the normalized harmonic current $\overline{T}$ giving mass $1$ to the generic fibers. 
Proposition \ref{p:weaker} 
 holds without modification, as well as   the computation of the index $I$ made in \S \ref{sec:proof} 
 (which is local near the punctures), so we infer that  
$$\chi (\sigma)  = \displaystyle \frac{1}{2\abs{\mathrm{eu}(X)} } ( N_{\overline{\mathcal F}} \cdot \overline{T} + \#P). $$

Now if $\mathcal G$ is a singular holomorphic foliation on a complex surface, and $C $ is a non singular compact holomorphic curve not everywhere tangent to $\mathcal G$, and not intersecting the singular set of $\mathcal G$, we have 
\[ N_{\mathcal G} \cdot C = \mathrm{eu} (C)  + | \mathrm{tang} (\mathcal G, C) |, \]
where the tangency points are counted with multiplicities (see \cite{brunella} for details).   
Hence formula \eqref{eq:npoint} has to be replaced by 
\begin{equation}\label{eq:npoint2}
N_{\overline{\mathcal{F} } }\cdot \overline s = \mathrm{eu} (\overline s) + |\mathrm{tang} (\overline{\mathcal F}, \overline{s} ) |= \mathrm{eu} (\overline X) + k \text{ and } 
N_{\overline{\mathcal{F}}}\cdot f = \mathrm{eu}(\pu) = 2.
\end{equation}
We also have 
\[  \overline{s}^2 = \mathrm{eu} (\overline{s}) + |\mathrm{tang} (\overline{\mathcal F}, \overline{s} ) | = \mathrm{eu} (\overline{s})  + k. \]
So we infer that 
$$[ N_{\overline{\mathcal F}} ] = 2 [\overline{s}] - (\mathrm{eu}(\overline{X} ) + k) [ f ] ,$$
and we conclude as in the proof of Theorem \ref{theo:formula}.
\end{proof}

\end{document}